\documentclass{amsart}

\usepackage{amssymb,mathrsfs}
\usepackage[T1]{fontenc}
\usepackage{enumitem}
\usepackage{tikz}
\usetikzlibrary{cd}
\usepackage{float}
\usepackage[pdfusetitle,colorlinks=true,linkcolor=magenta,citecolor=cyan,urlcolor=magenta]{hyperref}

\numberwithin{equation}{section}

\newtheorem{theorem}{Theorem}[section]
\newtheorem{proposition}[theorem]{Proposition}
\newtheorem{lemma}[theorem]{Lemma}
\newtheorem{corollary}[theorem]{Corollary}

\theoremstyle{remark}
\newtheorem{remark}[theorem]{Remark}

\DeclareMathOperator{\Isom}{Isom}
\DeclareMathOperator{\Spf}{Spf}
\DeclareMathOperator{\lcm}{lcm}

\newcommand{\Q}{\mathbb Q}
\newcommand{\Qbar}{\overline{\mathbb Q}}
\newcommand{\Z}{\mathbb Z}
\newcommand{\C}{\mathbb C}
\newcommand{\F}{\mathbb F}

\newcommand{\Hup}{\mathfrak H}
\newcommand{\Ga}{\mathbb G_{\mathrm a}}
\newcommand{\Gm}{\mathbb G_{\mathrm m}}
\newcommand{\SL}{\mathrm{SL}_2(\mathbb Z)}
\newcommand{\dd}{\,\mathrm d}
\newcommand{\Xmu}{X_\mu(N)}
\newcommand{\Gpm}{\Gamma_1^{\pm}(N)}
\newcommand{\cE}{\mathscr E}

\newcommand{\diam}[1]{\langle #1\rangle}

\newcommand{\BDC}{\hyperlink{bdc}{\textup{(BDC)}}}
\newcommand{\sm}[2]{\bigl(\begin{smallmatrix}#1\\#2\end{smallmatrix}\bigr)}

\theoremstyle{definition}
\newtheorem{definition}[theorem]{Definition}

\newcommand{\Ymu}{Y_\mu(N)}
\DeclareMathOperator{\Incl}{Incl}
\DeclareMathOperator{\Ig}{Ig}

\title{On the Manin Constant}
\author{Zhuoni Chi}
\email{znchi@zju.edu.cn}
\author{Shurui Liu}
\email{srliu@stanford.edu}
\author{Ping Nie}
\email{ping.nie@pku.edu.cn}
\hypersetup{pdfauthor={Zhuoni Chi, Shurui Liu, and Ping Nie}}
\date{}

\begin{document}

\begin{abstract}
We prove the Manin constant conjecture for every optimal elliptic curve over
$\Q$ and Stevens' Conjectures~I and~II. Our method adapts Hongbo Yin's
congruence argument for cubic roots of modular functions to lifts of modular
parametrizations through isogenies, using the unbounded denominators theorem
of Calegari, Dimitrov and Tang.
\end{abstract}

\maketitle

\setcounter{tocdepth}{1}
\tableofcontents

% !TEX root = ../paper.tex
\section{Introduction}\label{sec:intro}

\subsection{Main results}\label{subsec:results}

Let $E$ be an elliptic curve over $\Q$ of conductor $N$, and let
$f=\sum_{n\geq1}a_nq^n$ be the newform of weight two and level $N$ attached to
$E$. Here $q=e^{2\pi iz}$, with $z$ in the upper half-plane $\Hup$. We write
\[
  \omega_f=2\pi if(z)\dd z=f(q)\frac{\dd q}{q}
\]
for the corresponding differential on the modular curve $X_0(N)$. A
surjection $\pi_0:J_0(N)\twoheadrightarrow E$ is called \emph{optimal} if its
kernel is connected. Each isogeny class contains exactly one curve admitting
an optimal quotient, and $\pi_0$ is then unique up to sign. Let
\[
  \varphi_0:X_0(N)\longrightarrow E,\qquad \varphi_0(P)=\pi_0([P]-[\infty])
\]
be the corresponding modular parametrization, and let $\omega_E$ be a N\'eron
differential of $E$. By multiplicity one there is a constant $c_E\in\Q^\times$,
the \emph{Manin constant}, with
\[
  \varphi_0^*\omega_E=c_E\,\omega_f .
\]
Manin conjectured that $c_E=\pm1$ \cite[Section~10.3]{Manin1971}; see
\cite[Conj.~1.1]{Cesnavicius2018} for the formulation used here.

\begin{theorem}[Manin constant conjecture]\label{thm:manin}
Let $E$ be an elliptic curve over $\Q$ of conductor $N$, and let
$\pi_0:J_0(N)\twoheadrightarrow E$ be an optimal quotient. Then $c_E=\pm1$.
\end{theorem}
We will prove this conjecture in Section~\ref{sec:deduction} under some assumptions and remove the assumptions in Sections~\ref{sec:integrality} and~\ref{sec:lifting}. The case $|c_E|=2$ will be excluded in Section~\ref{sec:c-two} because the techniques there are heavier.

Stevens studied the analogous question for $X_1(N)$
\cite{Stevens1989}. Here $X_1(N)$ denotes Shimura's canonical model over $\Q$,
in which the cusp $0$ is a rational point \cite[Section~1]{Stevens1989}. Let
$\mathcal C$ be an isogeny class of elliptic curves over $\Q$ of conductor $N$.
There is a unique $E_1\in\mathcal C$ admitting a parametrization
$\varphi_1:X_1(N)\to E_1$ which sends the cusp $0$ to $0$ and whose
Jacobian map
$\pi_1:J_1(N)\twoheadrightarrow E_1$ has connected kernel. By \cite[Prop.~1.4]{Stevens1989}, $\varphi_1$ is
unique up to sign. Its Manin constant $c_1$ is
defined by
\[
  \varphi_1^*\omega_{E_1}=c_1\,\omega_f^{\mathrm{St}},
\]
where $\omega_{E_1}$ is a N\'eron differential and $\omega_f^{\mathrm{St}}$ is
the differential attached to $f$ in \cite[Def.~1.1]{Stevens1989} (see
Section~\ref{subsec:proof-stevens}). Stevens proved that $c_1\in\Z$
\cite[Thm.~1.6]{Stevens1989}. In each isogeny class $\mathcal C$ there is a unique curve $E_{\min}$
such that, for every $E\in\mathcal C$, there is an isogeny
$E_{\min}\to E$ over $\Q$ which extends to an \'etale morphism of
N\'eron models over $\Z_p$ for every prime $p$
\cite[Thm.~2.3]{Stevens1989}. He also conjectured the following:

\begin{theorem}[Stevens' Conjectures I and II]\label{thm:stevens}
For every isogeny class of elliptic curves over $\Q$ we have $c_1=\pm1$ and
$E_1\simeq E_{\min}$.
\end{theorem}

The two assertions are Stevens' Conjecture~I$'$(a) and Conjecture~II,
respectively. Together they imply his Conjecture~I
\cite[(1.7), Thms.~1.8 and~2.3, (2.4)]{Stevens1989}.
The first assertion follows from Theorem~\ref{thm:manin} by
\cite[Lemma~2.12]{Cesnavicius2018}. The second will be deduced from
Theorem~\ref{thm:manin} and Theorem~\ref{thm:lifting} below in
Section~\ref{subsec:proof-stevens}.

\subsection{The normalized lift}\label{subsec:normalized-lift}
The central object of this article is the lift of a modular parametrization
through an isogeny. Let $\nu:E'\to E$ be an isogeny and let $\varphi$ be a
morphism from a modular curve to $E$. Since $\Hup$ is simply connected,
$\varphi$ lifts through $\nu$ to a holomorphic map $\Hup\to E'(\C)$. We first fix the
modular curves on which $\varphi$ is defined.

For a subgroup $G\subseteq\SL$ we write $\pm G=\langle G,-I\rangle$, and we
put $\Gpm=\pm\Gamma_1(N)$.

Next we introduce the modular curve $\Xmu$ we use in this article.
\begin{definition}\label{def:Xmu}
Let $\Ymu$ be the coarse moduli scheme over $\Q$ of pairs
\[
 (A/W,\ \iota),
\]
where
\begin{itemize}[label=$\circ$,leftmargin=1.8em,itemsep=3pt,topsep=4pt]
\item $W$ is a $\Q$-scheme and $A/W$ is an elliptic curve, and
\item $\iota:\mu_{N,W}\hookrightarrow A[N]$ is a homomorphism of $W$-group
schemes which is a closed immersion.
\end{itemize}
An isomorphism from $(A/W,\iota)$ to $(A'/W,\iota')$ is an isomorphism
$g:A\to A'$ over $W$ such that $g\circ\iota=\iota'$. Let $\Xmu$ be the
compactification of $\Ymu$, that is, the normalization of the projective
$j$-line $\mathbb P^1_\Q$ in $\Ymu$. The \emph{cusps} of $\Xmu$ are the points
of $\Xmu-\Ymu$.\footnote{In the notation of \cite{KM1985},
$\Ymu=M(\mathcal I_N\otimes_\Z\Q)$ and
$\Xmu=\overline M(\mathcal I_N\otimes_\Z\Q)$, where $\mathcal I_N$ is the
moduli problem
$E/S\mapsto\Incl_{S\text{-gp}}(\mu_N,E[N])$ of \cite[Section~4.9]{KM1985}, and
$M$ and $\overline M$ are defined in \cite[(8.1.1) and (8.6.3)]{KM1985}.}
\end{definition}

\begin{remark}\label{rem:Xmu}
Since $\Ymu$ is a smooth curve over $\Q$, the curve $\Xmu$ is the unique
smooth projective curve over $\Q$ containing $\Ymu$ as a dense open
subscheme. The $j$-invariant extends to a morphism
$j:\Xmu\to\mathbb P^1_\Q$, and $\Xmu-\Ymu=j^{-1}(\infty)$
\cite[(8.6.3.2)]{KM1985}. Thus $\Xmu$ is obtained from $\Ymu$ by adding the
cusps, which are the poles of $j$.
\end{remark}

Let $\infty$ be the cusp of
$\Xmu$ corresponding to the Tate curve $\mathrm{Tate}(q)$ over $\Q((q))$ with
$\mu_N\subseteq\Gm$. It is a rational point, and the completion of $\Xmu$
along it is $\Spf\Q[[q]]$ \cite[Thm.~8.11.10]{KM1985}. In terms of
generalized elliptic curves, $\infty$ is given by the standard N\'eron
$1$-gon with $\mu_N$ embedded in its smooth locus $\Gm$
\cite[proof of Lemma~6.1.6]{CES}. Analytically we use the identification
\[
  \Gpm\backslash\Hup^*\simeq\Xmu(\C),\qquad
  z\longmapsto\bigl(\C/(\Z z+\Z),\ e^{2\pi i/N}\mapsto1/N\bigr),
\]
under which $\infty$ corresponds to $\infty$ and $q=e^{2\pi iz}$.

For $\delta\in(\Z/N\Z)^\times$ the diamond operator
$\diam\delta:(A,\iota)\mapsto(A,\iota\circ[\delta])$ is an automorphism of
$\Xmu$ over $\Q$, and $\diam{-1}$ is the identity. A matrix
$\gamma=\sm{a&b}{Nr&d}\in\Gamma_0(N)$ acts on $\Xmu(\C)$ as $\diam d$: the map
$w\mapsto(Nrz+d)w$ is an isomorphism $\C/(\Z\gamma z+\Z)\to\C/(\Z z+\Z)$
sending $1/N$ to $d/N$ modulo $\Z z+\Z$.

Let $\Gamma$ be a group with $\Gpm\subseteq\Gamma\subseteq\Gamma_0(N)$, let
$\Delta_\Gamma$ be the image of $\Gamma$ under the homomorphism
$\Gamma_0(N)\to(\Z/N\Z)^\times/\{\pm1\}$, $\sm{a&b}{Nr&d}\mapsto\pm d$, and
put
\[
  X_\Gamma=\Xmu/\Delta_\Gamma,\qquad
  \varpi_\Gamma:\Xmu\longrightarrow X_\Gamma .
\]
Then $X_\Gamma(\C)=\Gamma\backslash\Hup^*$, and $X_\Gamma$ has the rational
cusp $\infty$ with parameter $q$. The cusps $\diam\delta\infty$ of $\Xmu$ are
rational. For $\Gamma=\Gamma_0(N)$ this is the usual model of $X_0(N)$, and we
write $\varpi=\varpi_{\Gamma_0(N)}:(A,\iota)\mapsto(A,\iota(\mu_N))$. We also
write $\omega_f$ for the pullbacks of $\omega_f$ to $X_\Gamma$ and to $\Xmu$.
Their expansion at $\infty$ is again $f(q)\dd q/q$.

Let $\nu:E'\to E$ be an isogeny of elliptic curves over $\Q$, and let
$\varphi:X_\Gamma\to E$ be a non-constant morphism over $\Q$ with
$\varphi(\infty)=0$.

\begin{definition}\label{def:normalized-lift}
The \emph{normalized lift} of $\varphi$ through $\nu$ is the holomorphic map
$\psi:\Hup\to E'(\C)$ with
\[
  \nu\circ\psi=\varphi\quad\text{on }\Hup,
  \qquad \psi(z)\to0\ \text{as}\ z\to i\infty.
\]
In other words, $\psi$ makes the following diagram commute:
\[
\begin{tikzcd}
  \Hup \arrow[r, dashed, "\psi"] \arrow[d] & E'(\C) \arrow[d, "\nu"] \\
  X_\Gamma(\C) \arrow[r, "\varphi"] & E(\C)
\end{tikzcd}
\]
\end{definition}

\begin{lemma}\label{lem:lift}
The normalized lift exists and is unique. It satisfies $\psi(z+1)=\psi(z)$,
and as a function of $q=e^{2\pi iz}$ it extends holomorphically to $q=0$,
with value $0$ there.
\end{lemma}

\begin{proof}
The map $\nu:E'(\C)\to E(\C)$ is a finite covering and $\Hup$ is simply
connected, so $\varphi$ lifts to a holomorphic map $\Hup\to E'(\C)$. If
$\psi_1$ and $\psi_2$ are
two lifts, then $\psi_1-\psi_2$ is a continuous map from the connected space
$\Hup$ to the finite set $\ker\nu$, hence constant.

Choose a disc $D\subseteq E(\C)$ around $0$ over which $\nu$ is a trivial
covering, and $y_0>0$ with $\varphi(z)\in D$ for $\operatorname{Im}z>y_0$. Let
$\psi_1$ be a lift. Since the
half-plane $\{\operatorname{Im}z>y_0\}$ is connected, $\psi_1$ maps it into a
single component of $\nu^{-1}(D)$. This component maps homeomorphically onto
$D$ and contains exactly one point $\kappa\in\ker\nu$, so $\psi_1(z)\to\kappa$
as $z\to i\infty$. Then $\psi=\psi_1-\kappa$ is a lift with $\psi(z)\to0$. If
$\psi'$ is another such lift, the constant $\psi'-\psi\in\ker\nu$ equals the
limit of $\psi'(z)-\psi(z)$ as $z\to i\infty$, which is $0$. Hence the
normalized lift exists and is unique.

The function $z\mapsto\psi(z+1)-\psi(z)$ is continuous with values in
$\ker\nu$, hence constant, and its limit is zero as $z\to i\infty$. Thus
$\psi(z+1)=\psi(z)$, and $\psi$ is a holomorphic function of $q$ on a
punctured disc which tends to $0$ as $q\to0$. It therefore extends across
$q=0$.
\end{proof}

We will study the geometry and arithmetic of this analytic map $\psi$. If
$x,y$ are the coordinate functions of a global minimal Weierstrass equation
for $E'$, then $x(\psi(q))$ and $y(\psi(q))$ lie in $\Q((q))$
(this is an easy fact and we will prove it in Lemma~\ref{lem:lift-expansion}).

\begin{lemma}\label{lem:rho}
For $\gamma\in\Gamma$ the function $z\mapsto\psi(\gamma z)-\psi(z)$ is a
constant $\rho(\gamma)\in\ker\nu$. The map $\rho:\Gamma\to\ker\nu$ is a
homomorphism, and its kernel
\[
  G_\psi=\ker\rho=\{\gamma\in\Gamma:\psi\circ\gamma=\psi\}
\]
is a normal subgroup of finite index in $\Gamma$.
\end{lemma}

\begin{proof}
Since $\varphi$ is $\Gamma$-invariant, $\psi\circ\gamma$ is also a lift of
$\varphi$, so the difference is a continuous function with values in the
finite group $\ker\nu$, hence constant. For $\gamma,\gamma'\in\Gamma$,
\[
  \rho(\gamma\gamma')
  =\bigl(\psi(\gamma\gamma'z)-\psi(\gamma'z)\bigr)
   +\bigl(\psi(\gamma'z)-\psi(z)\bigr)
  =\rho(\gamma)+\rho(\gamma').
\]
The last assertion follows because $\ker\nu$ is finite.
\end{proof}

The most important result of this article is the following lifting theorem, which is proved in Section~\ref{sec:lifting}:
\begin{theorem}\label{thm:lifting}
Suppose that the \emph{bounded denominators condition}
\hypertarget{bdc}{\textup{(BDC)}} holds for
$\psi$: there is an integer $k\geq1$ with
\[
  k\,x(\psi(q)),\ k\,y(\psi(q))\in\Z((q)).
\]
Then the following hold.
\begin{enumerate}[label=\textup{(\roman*)}]
\item\label{lift:congruence} \textup{\textbf{(Modularity)}} $G_\psi$ is a congruence subgroup and is normal
in $\Gamma$. The compact modular curve $X_{G_\psi}(\C)$ is the connected
component of $X_\Gamma(\C)\times_{\varphi,E,\nu}E'(\C)$ through $(\infty,0)$.
\item\label{lift:modular} \textup{\textbf{(Level)}} $\Gpm\subseteq G_\psi\subseteq\Gamma$, and $G_\psi$
contains every parabolic and elliptic element of $\Gamma$.
\item\label{lift:rational} \textup{\textbf{(Rationality)}} There is a unique morphism $\psi_\mu:\Xmu\to E'$
over $\Q$ induced from $\psi$ with $\psi_\mu(\infty)=0$. $\psi_\mu$
satisfies $\nu\circ\psi_\mu=\varphi\circ\varpi_\Gamma$, where
$\varpi_\Gamma:\Xmu\to X_\Gamma$ is the quotient map. In other words, the
following diagram commutes:
\[
\begin{tikzcd}
  \Xmu \arrow[r, dashed, "\psi_\mu"] \arrow[d, "\varpi_\Gamma"'] & E' \arrow[d, "\nu"] \\
  X_\Gamma \arrow[r, "\varphi"] & E
\end{tikzcd}
\]

The monodromy group
\[
  H_\psi=\{\psi(\gamma z)-\psi(z):\gamma\in\Gamma\}\subseteq\ker\nu
\]
consists of the points $\tau_\delta=\psi_\mu(\diam\delta\infty)$, where
$\diam\delta$ runs through the diamond operators in $\Gamma/\Gpm$. Moreover, these
points are rational, $H_\psi\subseteq E'(\Q)$, and
$\psi_\mu\circ\diam\delta=\psi_\mu+\tau_\delta$.

The component in \textup{(i)} descends to $\Q$, and its deck transformation group is the
constant group $H_\psi$, which is a quotient of $\Gamma/\Gamma'$, where
$\Gamma'$ is the subgroup generated by $\Gpm$ and the parabolic and elliptic
elements of $\Gamma$.
\end{enumerate}
\end{theorem}

Theorem~\ref{thm:manin} uses the case $\nu=[c]$ where $c=|c_E|$, and
Theorem~\ref{thm:stevens} uses the case of the \'etale isogeny
$E_{\min}\to E_1$. 

We will prove the following proposition in Section~\ref{sec:integrality} to show that \BDC{} holds with $k=1$ in both cases:

\begin{proposition}\label{prop:integrality}
Let $V$ be $X_0(N)$ or $\Xmu$, let $E$ be an elliptic curve over $\Q$ of
conductor $N$ with newform $f$, and let $\varphi:V\to E$ be a morphism over
$\Q$ with $\varphi(\infty)=0$.
\begin{enumerate}[label=\textup{(\roman*)}]
\item\label{int:root} If $\varphi^*\omega_E=c\,\omega_f$ for an integer
$c\geq1$ such that $p^2\mid N$ for every prime $p\mid c$, then the normalized
lift of $\varphi$ through $[c]$ satisfies \BDC{} with $k=1$.
\item\label{int:etale} If $\varphi^*\omega_E=\omega_f$ and $\nu:E'\to E$ is an
isogeny over $\Q$ with $\nu^*\omega_E=\omega_{E'}$, then the normalized lift
of $\varphi$ through $\nu$ satisfies \BDC{} with $k=1$.
\end{enumerate}
\end{proposition}

Part \ref{int:root} also follows from Honda's comparison of formal groups
\cite[Thm.~9]{Honda1970}, but we give a direct proof in Section~\ref{sec:integrality}.

\subsection{Outline of this article}\label{subsec:strategy}

In Section~\ref{sec:deduction}, we deduce Theorems~\ref{thm:manin}
and~\ref{thm:stevens} from Theorem~\ref{thm:lifting},
Proposition~\ref{prop:integrality} and Proposition~\ref{prop:c-one}, whose
proofs are postponed to Sections~\ref{sec:integrality}--\ref{sec:c-two}.
Proposition~\ref{prop:integrality} gives \BDC{} for the two lifts we need, so
that Theorem~\ref{thm:lifting} applies to them. For the lift of $\varphi_0$
through $[c]$, optimality then gives $E[c]\subseteq E(\Q)$, and the Weil
pairing gives $c\mid2$. Proposition~\ref{prop:c-one} excludes $c=2$. Stevens'
conjectures follow from Theorem~\ref{thm:manin} together with
Theorem~\ref{thm:lifting} for the lift through $E_{\min}\to E_1$. We put the
deduction first so that the reader can see how Theorem~\ref{thm:lifting} is
used, and why it suffices for the Manin constant conjecture, before the
technical details.

In Section~\ref{sec:integrality}, we prove that the normalized lift $\psi$ has integral $q$-expansions and \BDC{} holds. The integrality of the modular parametrization (Lemma~\ref{lem:integral-param}) follows \cite[proof of Prop.~2]{Edixhoven1991} and \cite[proof of Thm.~1.6]{Stevens1989}. The integrality of the lift through $[c]$ can be traced back to Honda's result \cite[Thm.~9]{Honda1970}, and we give a direct proof in the additive reduction case in Section~\ref{sec:integrality}.

In Section~\ref{sec:lifting}, we prove Theorem~\ref{thm:lifting}, which is the heart of this article.  Here Proposition~\ref{prop:integrality} gives integral expansions of the normalized lift first, so \BDC{} holds with $k=1$. Theorem~\ref{thm:lifting} then turns this integrality into modularity and rationality of the lift.

In Section~\ref{sec:c-two}, we exclude the case $|c_E|=2$. For a prime
$p\mid N$ we use the smooth model of $\Ymu$ over $\Z_p$ given by Katz and
Mazur, on which the diamond operators act. Its geometric special fibre is the
image of an open subscheme of an Igusa curve, and hence connected. This forces the points
$\tau_\delta$ of Theorem~\ref{thm:lifting}\ref{lift:rational} into the
identity component of the N\'eron model of $E$ at $p$.
We apply this at a prime $p$ of bad reduction, which exists by a theorem of
Tate. For a prime $\ell\mid c$ this is impossible: since the reduction at
$p$ is bad, the identity component has at most $\ell$ $\Z_p$-points killed
by $\ell$, whereas $E[\ell]$ has $\ell^2$ points.
Section~\ref{sec:c-two} should be considered as part of
Section~\ref{sec:deduction}, but we separate it because the techniques are much more complicated. This section works for every prime $\ell\mid c$ but we only need $\ell=2$ for the purpose of this article.

Figure~\ref{fig:structure} summarizes the logical dependencies among these sections.

\begin{figure}[ht]
\centering
\begin{tikzpicture}[
  font=\footnotesize,
  >=stealth,
  box/.style={draw, rounded corners=2pt, align=center, inner sep=3pt,
    text width=3.3cm},
  sub/.style={draw, rounded corners=2pt, align=center, inner sep=3pt,
    text width=2.9cm, fill=black!5},
  arr/.style={->, semithick},
  lab/.style={font=\scriptsize, inner sep=1.5pt},
  sec/.style={font=\scriptsize, text=black!60}
]
% The lifting theorem
\node[sub, text width=3.45cm] (L1) at (1.1,0)
  {\textbf{(i) Modularity}\\ $G_\psi$ is congruence\\ {\scriptsize(Calegari--Dimitrov--Tang)}};
\node[sub] (L2) at (4.85,0)
  {\textbf{(ii) Level}\\ $\Gpm\subseteq G_\psi$\\ {\scriptsize(Wohlfahrt)}};
\node[sub] (L3) at (8.4,0)
  {\textbf{(iii) Rationality}\\ $\psi_\mu$ over $\Q$, $H_\psi\subseteq E'(\Q)$\\
   {\scriptsize(Galois descent)}};
\draw[arr] (L1) -- (L2);
\draw[arr] (L2) -- (L3);
\draw[rounded corners=3pt] (-0.9,-0.85) rectangle (10.1,1.25);
\node[anchor=north west, inner sep=2pt] at (-0.9,1.25)
  {Section~\ref{sec:lifting}: Theorem~\ref{thm:lifting}, assuming \BDC{}};

% The Manin constant
\node[box] (A) at (0.9,-2.2)
  {{\scriptsize Section~\ref{sec:integrality},
    Proposition~\ref{prop:integral-root}}\\
   Proposition~\ref{prop:integrality}\ref{int:root}\\
   \BDC{} for $\nu=[c]$};
\node[box] (D) at (8.1,-2.2)
  {{\scriptsize Section~\ref{sec:deduction}, Lemma~\ref{lem:optimal}}\\
   optimality: $H_\psi=E[c]$};
\node[box] (E) at (4.5,-3.9)
  {{\scriptsize Section~\ref{sec:deduction},
    Proposition~\ref{prop:rational-torsion}}\\
   $E[c]=\{\tau_\delta\}\subseteq E(\Q)$};
\draw[arr] (4.5,-0.85) -- node[lab, right] {$\nu=[c]$} (E);
\draw[arr] (A) -- (E);
\draw[arr] (D) -- (E);

\node[box] (F) at (2.0,-5.7)
  {{\scriptsize Section~\ref{sec:deduction}}\\ $c\mid2$ (Weil pairing)};
\node[box, text width=3.6cm] (G) at (7.0,-5.7)
  {{\scriptsize Section~\ref{sec:c-two}, Proposition~\ref{prop:c-one}}\\
   $c\neq2$ by local torsion};
\draw[arr] (E) -- (F);
\draw[arr] (E) -- (G);

\node[box, very thick] (H) at (4.5,-7.5)
  {{\scriptsize Section~\ref{sec:deduction}}\\ Theorem~\ref{thm:manin}\\ $c_E=\pm1$};
\draw[arr] (F) -- (H);
\draw[arr] (G) -- (H);

% Stevens' conjectures
\node[box] (B) at (0.9,-9.3)
  {{\scriptsize Section~\ref{sec:integrality},
    Proposition~\ref{prop:integral-etale}}\\
   Proposition~\ref{prop:integrality}\ref{int:etale}\\
   \BDC{} for $\nu=\alpha$};
\node[box, very thick] (I) at (7.0,-9.3)
  {{\scriptsize Section~\ref{sec:deduction}, Lemma~\ref{lem:jacobian}}\\
   Theorem~\ref{thm:stevens}\\ $c_1=\pm1$, $E_1\simeq E_{\min}$};
\draw[arr] (H) -- (B);
\draw[arr] (H) -- (I);
\draw[arr] (B) -- (I);
\draw[arr] (10.1,0.2) -- (10.45,0.2) |- node[lab, pos=0.25, right] {$\nu=\alpha$}
  (I.east);
\end{tikzpicture}
\caption[Logical structure of Sections~\ref{sec:deduction}--\ref{sec:c-two}]{Logical structure of Sections~\ref{sec:deduction}--\ref{sec:c-two}.
For example, the lower-left block means that
Proposition~\ref{prop:integrality}\ref{int:etale} is proved as
Proposition~\ref{prop:integral-etale} in Section~\ref{sec:integrality}, and
gives \BDC{} for $\nu=\alpha$.
An arrow means that the target is deduced from the source.}
\label{fig:structure}
\end{figure}
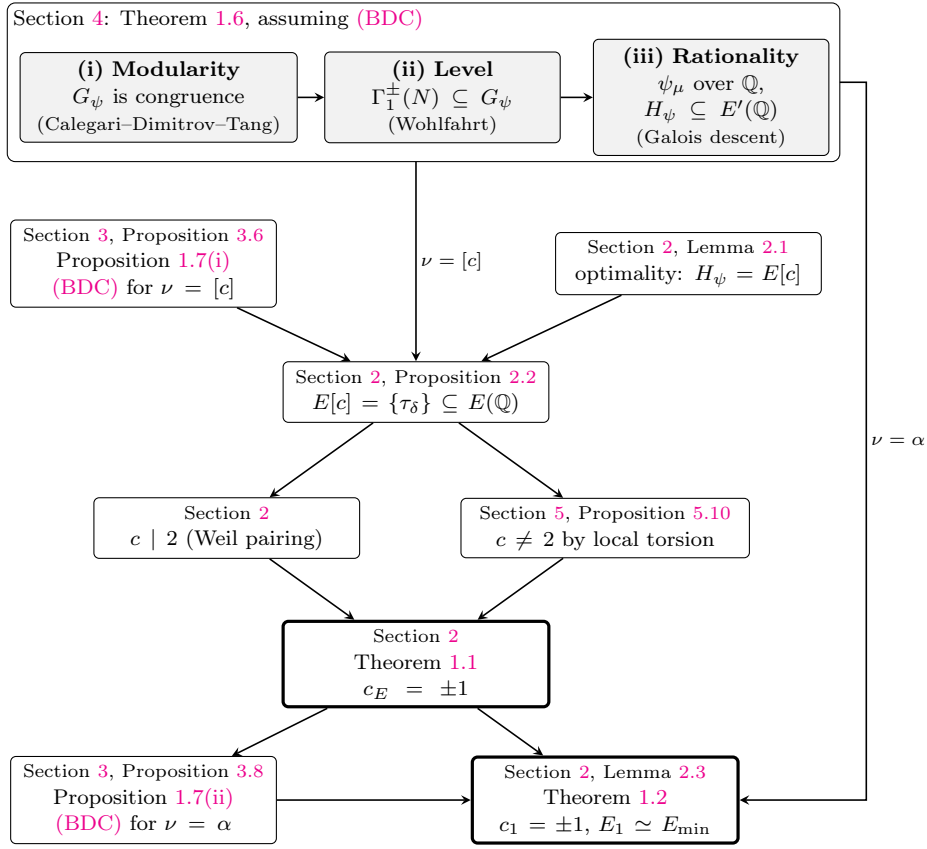

\subsection{What's new in this article}\label{subsec:new}

The study of the Manin constant has a long history. Edixhoven proved that $c_E$
is an integer \cite[Prop.~2]{Edixhoven1991}, and \v{C}esnavi\v{c}ius proved
that $\operatorname{ord}_p(c_E)=0$ whenever $p^2\nmid N$
\cite[Thm.~2.11]{Cesnavicius2018}. In particular the conjecture holds for
semistable $E$. For earlier special cases see \cite{Mazur1978},
\cite{AbbesUllmo1996}, \cite{ARS2006}, \cite{CNS} and the survey in
\cite[Section~1]{Cesnavicius2018}. For Stevens' conjectures, Stevens checked
conductors up to $200$ \cite[Thm.~7.1]{Stevens1989} and Vatsal treated some
further cases \cite{Vatsal2005}. These methods are local: they study the
N\'eron model of $J_0(N)$ at $p$, and they run into difficulty when
$p^2\mid N$, that is, when $E$ has additive reduction at $p$. The integrality of the coordinate functions of $\psi$ is closely related to Honda's theorem on formal groups \cite[Thm.~9]{Honda1970}. Lifts through multiplication maps on
a particular modular curve are also studied in \cite{BFF2023}.

Our starting point is more global. As $\Hup$ is simply connected, the parametrization $\varphi_0$ should have a ``$c$-th root''. So we take the normalized lift $\psi$ of $\varphi_0$ through $[c]$. A
priori $\psi$ is only a holomorphic map on $\Hup$, and there is no reason for
it to be invariant under any congruence subgroup. 

However, its $q$-expansion is easily seen to be integral and we may apply the unbounded denominators theorem of Calegari, Dimitrov and Tang \cite{CDT}. Together with an argument from \cite{Yin2026}, this theorem forces $\psi$ to be modular, and we use Galois descent to show the induced algebraic map from the modular curve to $E$ comes from a morphism $\psi_\mu:\Xmu\to E$ over $\Q$. Optimality now forces $E[c]\subseteq E(\Q)$, and so $c\mid2$. The case $c=2$ needs a separate local argument, using the Igusa curve (Section~\ref{sec:c-two}) or $2$-isogenies and the cusps of $\Xmu$ (Appendix~\ref{app:another}). The same argument, applied to the \'etale isogeny $E_{\min}\to E_1$ instead of $[c]$, gives Stevens' Conjecture~II.

The main innovation of this paper is therefore Theorem~\ref{thm:lifting},
which lets us treat the normalized lift $\psi$ as an arithmetic object rather
than just a holomorphic map. The idea comes from Yin's work on cubic roots of
modular functions attached to CM elliptic curves \cite[Prop.~8.1]{Yin2026},
where the unbounded denominators theorem is used to show that these cubic
roots are invariant under a congruence subgroup. There the integral
expansions required by \cite{CDT} are deduced from the fact that the relevant
Manin--Stevens constants are units \cite[Thm.~6.1]{Yin2026}. We go the other
way: we prove integrality first, and then use it to show that the lift is
modular and to bound $c$.

\subsection{Acknowledgements}
Z.\,Chi would like to thank Hongbo Yin for sharing the motivation behind the
beautiful result in his paper and for his encouragement of this work. S.\,Liu would like to thank Richard Taylor and Brian Conrad for inspiring conversations and suggestions on the manuscript. Z.\,Chi, S.\,Liu and P.\,Nie thank Yiyang Wang for meaningful conversations. Z.\,Chi would like to thank his advisor Yifeng Liu, and S.\,Liu would like to thank his advisor Xinwen Zhu for constant support.

%Z.\,Chi and S.\,Liu are deeply grateful to Richard Taylor and Brian Conrad
% for suggestions on improving the manuscript.
% !TEX root = ../paper.tex
\section{Deduction of the main theorems}\label{sec:deduction}

In this section we deduce Theorems~\ref{thm:manin} and~\ref{thm:stevens}
from Theorem~\ref{thm:lifting} and Proposition~\ref{prop:integrality}, which
are proved in Sections~\ref{sec:lifting} and~\ref{sec:integrality}. Together
with optimality, they give $E[c]\subseteq E(\Q)$
(Proposition~\ref{prop:rational-torsion}), and the Weil pairing then gives
$c\mid2$. The proof of Theorem~\ref{thm:manin} is completed in
Section~\ref{sec:c-two}, where Proposition~\ref{prop:c-one} uses
Proposition~\ref{prop:rational-torsion} and the geometry of Igusa curves to
show that $c=1$. Stevens' conjectures are then deduced from
Theorem~\ref{thm:manin} in Section~\ref{subsec:proof-stevens}.

\subsection{Proof of the Manin constant conjecture}\label{subsec:proof-manin}

Let $E$, $\pi_0$ and $\varphi_0$ be as in Theorem~\ref{thm:manin}. Replacing
$\pi_0$ by $-\pi_0$, which changes the sign of $c_E$ and preserves
optimality, we may assume that $c=c_E$ is positive. By
\cite[Prop.~2]{Edixhoven1991}, $c$ is an integer. \v{C}esnavi\v{c}ius proved
that $\operatorname{ord}_p(c)=0$ for every prime $p$ at which $E$ is
semistable, that is, with $p^2\nmid N$ \cite[Thm.~2.11]{Cesnavicius2018}.
This settles Theorem~\ref{thm:manin} when $N$ is squarefree. Thus
\begin{equation}\label{eq:support}
  c\in\Z_{\geq1},\qquad p^2\mid N\ \text{for every prime }p\mid c .
\end{equation}

We apply the constructions of Section~\ref{subsec:normalized-lift} with
$\Gamma=\Gamma_0(N)$ and $\varphi=\varphi_0$. The input from optimality is
the following.

\begin{lemma}\label{lem:optimal}
Let $m\geq1$, and let $\psi$ be the normalized lift of $\varphi_0$ through
$[m]:E\to E$. Recall from Theorem~\ref{thm:lifting}\ref{lift:rational} its
monodromy group
\[
  H_\psi=\{\psi(\gamma z)-\psi(z):\gamma\in\Gamma_0(N)\}
  =\rho(\Gamma_0(N))\subseteq E[m].
\]
Then $H_\psi=E[m]$.
\end{lemma}

\begin{proof}
Since $\ker\pi_0$ is connected, the homotopy sequence of
$0\to\ker\pi_0\to J_0(N)\to E\to0$ shows that
$\pi_{0*}:H_1(J_0(N)(\C),\Z)\to H_1(E(\C),\Z)$ is surjective. The
Abel--Jacobi map $\mathrm{AJ}:X_0(N)\to J_0(N)$, $P\mapsto[P]-[\infty]$,
induces an isomorphism
\[
  \mathrm{AJ}_*:H_1(X_0(N)(\C),\Z)\xrightarrow{\ \sim\ }H_1(J_0(N)(\C),\Z).
\]
Since $\varphi_0=\pi_0\circ\mathrm{AJ}$, the map
\[
  \varphi_{0*}:H_1(X_0(N)(\C),\Z)\longrightarrow H_1(E(\C),\Z)
\]
is also surjective. Let $X^\circ\subseteq X_0(N)(\C)$ be the complement of the cusps and of the
images of the elliptic points, and $\Hup^\circ\subseteq\Hup$ its preimage, so
that $\Hup^\circ\to X^\circ$ is a covering with group
$\Gamma_0(N)/\{\pm I\}$. Let $\ell$ be a loop in $X^\circ$ based at the image
of $z_0\in\Hup^\circ$. It lifts to a path in $\Hup^\circ$ from $z_0$ to
$\gamma z_0$ for some $\gamma\in\Gamma_0(N)$. Along this path, $\psi$ is a
continuous lift of $\varphi_0\circ\ell$ through $[m]$, and its endpoints
differ by $\rho(\gamma)$. On the other hand, under the identification
$E[m]=\tfrac1m\Lambda/\Lambda\simeq\Lambda/m\Lambda$, where
$\Lambda=H_1(E(\C),\Z)$, the endpoints of any lift of the loop
$\varphi_0\circ\ell$ through $[m]$ differ by the class of
$\varphi_{0*}[\ell]$. Hence this class is $\rho(\gamma)\in H_\psi$. Since
loops in $X^\circ$ generate $H_1(X_0(N)(\C),\Z)$, and $\varphi_{0*}$ is
surjective, $H_\psi$ contains, and therefore equals, $E[m]$.
\end{proof}

\begin{proposition}\label{prop:rational-torsion}
Let $\psi$ be the normalized lift of $\varphi_0$ through $[c]$. There are a
morphism $\psi_\mu:\Xmu\to E$ over $\Q$ and points $\tau_\delta\in E(\Q)$,
$\delta\in(\Z/N\Z)^\times$, with
\begin{equation}\label{eq:root-map}
  [c]\circ\psi_\mu=\varphi_0\circ\varpi,\qquad
  \psi_\mu\circ\diam\delta=\psi_\mu+\tau_\delta,\qquad
  \tau_\delta=\psi_\mu(\diam\delta\infty),
\end{equation}
and
\begin{equation}\label{eq:rational-torsion}
  \{\tau_\delta:\delta\in(\Z/N\Z)^\times\}=E[c]\subseteq E(\Q).
\end{equation}
\end{proposition}

\begin{proof}
By \eqref{eq:support} and
Proposition~\ref{prop:integrality}\ref{int:root}, applied with $V=X_0(N)$,
$\psi$ satisfies \BDC{} with $k=1$.
Theorem~\ref{thm:lifting}\ref{lift:rational}, with $\Gamma=\Gamma_0(N)$,
$\nu=[c]$ and $\varphi=\varphi_0$, gives $\psi_\mu$, and shows that the points
$\tau_\delta=\psi_\mu(\diam\delta\infty)$, $\delta\in(\Z/N\Z)^\times$, which are
rational, form $H_\psi$, and satisfy
$\psi_\mu\circ\diam\delta=\psi_\mu+\tau_\delta$. Finally $H_\psi=E[c]$ by
Lemma~\ref{lem:optimal}.
\end{proof}

\begin{proof}[Proof of Theorem~\textup{\ref{thm:manin}}]
By Proposition~\ref{prop:rational-torsion}, $E[c]\subseteq E(\Q)$. The Weil
pairing $E[c]\times E[c]\to\mu_c$ is surjective and Galois equivariant, so
$\mu_c\subseteq\Q$, and hence $c\mid2$. Finally $c\neq2$ will be proved in Section~\ref{sec:c-two} as
Proposition~\ref{prop:c-one}. Hence $c=1$, that is, $c_E=\pm1$.
\end{proof}

Thus the arguments of Sections~\ref{sec:integrality} and~\ref{sec:lifting} prove Theorem~\ref{thm:manin} at every odd prime, and Section~\ref{sec:c-two} is needed only to exclude an optimal curve with additive reduction at $2$ whose full
$2$-torsion is rational.

\subsection{Proof of Stevens' conjectures}\label{subsec:proof-stevens}

Let $\mathcal C$ be an isogeny class of conductor $N$ with newform $f$, and
let $\pi_1$, $\varphi_1$, $E_1$, $c_1$ and $E_{\min}$ be as in
Section~\ref{subsec:results}. Let $\pi_0:J_0(N)\twoheadrightarrow E$ be
the optimal quotient with $E\in\mathcal C$.

The curve $\Xmu$ is Shimura's canonical model of $X_1(N)$ used by Stevens: in
the proof of \cite[Thm.~1.6]{Stevens1989}, $X_1(N)_\Q$ is identified with the
$\mu_N$-moduli space, and his rational cusp $0$ with our cusp $\infty$ (see
also \cite[proof of Lemma~6.1.6]{CES}). His variable $\tau_{\mathrm{St}}$ and ours are related by
$z=-1/(N\tau_{\mathrm{St}})$, so that $e^{-2\pi i/(N\tau_{\mathrm{St}})}=q$. His
differential
$\omega_f^{\mathrm{St}}=2\pi if(\tau_{\mathrm{St}})\dd\tau_{\mathrm{St}}$ is
$\varepsilon_N\cdot2\pi if(z)\dd z$ in our coordinate, where
$\varepsilon_N=\pm1$ is the eigenvalue of the Fricke involution on $f$.
Hence, regarding $\varphi_1$ as a morphism $\Xmu\to E_1$ with
$\varphi_1(\infty)=0$,
\begin{equation}\label{eq:stevens-normalization}
  \varphi_1^*\omega_{E_1}=\varepsilon_Nc_1\,f(q)\frac{\dd q}{q}.
\end{equation}

\begin{lemma}\label{lem:jacobian}
Let $V$ be a smooth projective geometrically connected curve over a field of
characteristic zero with a rational point $o$, let
$\pi:J_V\twoheadrightarrow D_1$ be a quotient with connected kernel onto an
elliptic curve, and let $g:V\to D_2$ be a morphism to an elliptic curve with
$g(o)=0$. If $g_*(\ker\pi)$ is finite, where $g_*:J_V\to D_2$ is the induced
homomorphism, then $g_*=u\circ\pi$ for a unique homomorphism $u:D_1\to D_2$.
If moreover $\lambda\circ g_*=\pi$ for an isogeny $\lambda:D_2\to D_1$, then
$\lambda\circ u=\mathrm{id}_{D_1}$, and $\lambda$ and $u$ are isomorphisms.
\end{lemma}

\begin{proof}
The image $g_*(\ker\pi)$ is connected and finite, hence zero, so $g_*$ factors
through $J_V/\ker\pi=D_1$. If $\lambda g_*=\pi$, then $\lambda u\pi=\pi$, and
the surjectivity of $\pi$ gives $\lambda u=\mathrm{id}$. Taking degrees yields $\deg\lambda=\deg u=1$.
\end{proof}

\begin{proof}[Proof of Theorem~\textup{\ref{thm:stevens}}]
By \cite[Lemma~2.12]{Cesnavicius2018}, applied with $H=\Gamma_1(N)$ and
$H'=\Gamma_0(N)$, $c_1$ divides $c_E$, so $c_1=\pm1$ by
Theorem~\ref{thm:manin}.
Replacing $\varphi_1$ by $-\varphi_1$ if necessary, we may assume that
\begin{equation}\label{eq:stevens-unit}
  \varphi_1^*\omega_{E_1}=f(q)\frac{\dd q}{q}=\omega_f .
\end{equation}

It remains to prove that $E_1\simeq E_{\min}$. By
\cite[Thm.~2.3]{Stevens1989} there is an isogeny $\alpha:E_{\min}\to E_1$
which is \'etale on N\'eron models over $\Z_p$ for every $p$. By
\cite[Lemma~2.2]{Stevens1989} it pulls back a N\'eron differential to a
N\'eron differential at every $p$, so $\alpha^*\omega_{E_1}=\pm\omega_{E_{\min}}$,
and replacing $\alpha$ by $-\alpha$ we may assume that
$\alpha^*\omega_{E_1}=\omega_{E_{\min}}$. Let $\psi'$ be the normalized lift of
$\varphi_1$ through $\alpha$. By \eqref{eq:stevens-unit} and
Proposition~\ref{prop:integrality}\ref{int:etale}, applied with $V=\Xmu$,
$\psi'$ satisfies \BDC{} with $k=1$.

Theorem~\ref{thm:lifting}\ref{lift:rational}, with $\Gamma=\Gpm$,
$\varphi=\varphi_1$ and $\nu=\alpha$, now gives a morphism
$h':\Xmu\to E_{\min}$ over $\Q$ with $h'(\infty)=0$ and
$\alpha\circ h'=\varphi_1$. On Jacobians $\alpha\circ h'_*=\pi_1$, so
$h'_*(\ker\pi_1)\subseteq\ker\alpha$ is finite, and Lemma~\ref{lem:jacobian}
with $\pi=\pi_1$, $g=h'$ and $\lambda=\alpha$ shows that $\alpha$ is an
isomorphism.
\end{proof}

% !TEX root = ../paper.tex
\section{Integrality}\label{sec:integrality}

In this section we prove Proposition~\ref{prop:integrality}: part
\ref{int:root} is Proposition~\ref{prop:integral-root} and part
\ref{int:etale} is Proposition~\ref{prop:integral-etale}. The desired \BDC{} property follows from them.

For an elliptic curve $E$ over $\Q_p$ we write $\cE$ for its N\'eron model
over $\Z_p$, $\cE^0$ for its identity component, and
\[
  E(\Q_p)=\cE(\Z_p),\qquad E^0(\Q_p)=\cE^0(\Z_p).
\]
For a minimal Weierstrass equation we use the parameter $t_E=-x_E/y_E$ at the
origin, write $\widehat E$ for the formal group in this parameter, and
\[
  L_E(w)=w+\sum_{r\geq2}b_rw^r\in\Q_p[[w]]
\]
for its normalized logarithm. We omit the subscript $E$ when it is clear from
context. The smooth locus of the minimal Weierstrass model
is $\cE^0$ \cite[Cor.~IV.9.1]{Silverman1994}, and $\widehat E$ is its
completion along the zero section.

For an elliptic curve $E$ over $\Q$ we always use a global minimal Weierstrass
equation, and $\omega_E$ denotes its invariant differential which is a
N\'eron differential, and its expansion in $t$ is $\dd L_E(t)$.

For a map $g$ to $E$ with $g(\infty)=0$ we write
\begin{equation}\label{eq:coor-notation}
  t_g=t_E\circ g.
\end{equation}

\subsection{Formal logarithms at primes of additive reduction}\label{subsec:formal-log}

\begin{lemma}\label{lem:inverse}
Let $A$ be a commutative ring. A power series
$w+\sum_{r\geq2}\beta_rw^r\in A[[w]]$ has a compositional inverse in
$A[[w]]$.
\end{lemma}

\begin{proof}
Write the inverse as $w+\sum_{r\geq2}\gamma_rw^r$. The coefficient of $w^r$ in
the composite is $\gamma_r$ plus a polynomial with integer coefficients in
$\beta_2,\dots,\beta_r,\gamma_2,\dots,\gamma_{r-1}$, so the $\gamma_r$ are
determined recursively in $A$.
\end{proof}

\begin{lemma}\label{lem:formal-log}
Let $F$ be a one-dimensional commutative formal group law over $\Z_p$ with
$[p]_F(w)\in p\Z_p[[w]]$. Then the normalized logarithm $L_F$ and its
compositional inverse have coefficients in $\Z_p$.
\end{lemma}

\begin{proof}
Write $[p]_F(w)=p\,g(w)$ with $g(w)=w+\sum_{r\geq2}g_rw^r\in\Z_p[[w]]$, and
$L_F(w)=\sum_{r\geq1}b_rw^r$ with $b_1=1$. Comparing the coefficients of
$w^r$, $r\geq2$, in $L_F([p]_F(w))=pL_F(w)$ gives
\[
  (1-p^{r-1})\,b_r=g_r+\sum_{\ell=2}^{r-1}p^{\,\ell-1}b_\ell\,[w^r]\,g(w)^\ell ,
\]
where $[w^r]$ denotes the coefficient of $w^r$. Since $1-p^{r-1}$ is a unit
in $\Z_p$, induction on $r$ gives $b_r\in\Z_p$. The assertion about
$L_F^{-1}$ follows from Lemma~\ref{lem:inverse}.
\end{proof}

\begin{lemma}\label{lem:additive}
Let $E$ be an elliptic curve over $\Q_p$ with additive reduction. The formal
group of a minimal Weierstrass equation satisfies
$[p]_{\widehat E}(w)\in p\Z_p[[w]]$. Hence $L_E$ and $L_E^{-1}$ have
coefficients in $\Z_p$.
\end{lemma}

\begin{proof}
The formal group $\widehat E$ is the completion of $\cE^0$ along the zero
section. The special fibre $\cE^0_{\F_p}$ is a form
of $\Ga$ over the perfect field $\F_p$, hence isomorphic to $\Ga$, on which
multiplication by $p$ is zero. Therefore $[p]_{\widehat E}$ reduces to zero
modulo $p$, and Lemma~\ref{lem:formal-log} applies.
\end{proof}

\subsection{Integrality of the normalized lift}\label{subsec:integral-lift}

Let $\Gamma$, $\nu:E'\to E$ and $\varphi:X_\Gamma\to E$ be as in
Section~\ref{subsec:normalized-lift}, and let $\psi$ be the normalized lift
of $\varphi$ through $\nu$. We write $x,y$ for the coordinate functions of a
global minimal Weierstrass equation of $E'$, so that $t_{E'}=-x/y$. As in
\eqref{eq:coor-notation}, we put $t_\psi=t_{E'}\circ\psi$ and
$t_\varphi=t_E\circ\varphi$. Let $\widehat\nu(w)\in\Q[[w]]$ be the
expansion of $\nu$ in these parameters, so that
$t_E\circ\nu=\widehat\nu(t_{E'})$ near the origin. Its linear coefficient is the constant $d\in\Q^\times$ with
$\nu^*\omega_E=d\,\omega_{E'}$, which is nonzero because $\nu$ is \'etale.

\begin{lemma}\label{lem:lift-expansion}
We have $t_\psi=\widehat\nu^{-1}(t_\varphi)$, and this series is a nonzero
element of $q\,\Q[[q]]$. Consequently $x\circ\psi$ and $y\circ\psi$ are
Laurent series in $q$ with rational coefficients.
\end{lemma}

\begin{proof}
By Lemma~\ref{lem:lift}, $\psi$ is a holomorphic function of $q$ near $q=0$
with value $0$ at $q=0$, so $\widehat\nu(t_\psi)=t_\varphi$ near $q=0$. Since
$\varphi$, the cusp $\infty$ and the parameter $q$ are defined over $\Q$, we
have $t_\varphi\in q\,\Q[[q]]$, and $t_\varphi\neq0$ because $\varphi$ is not
constant. As $\widehat\nu$ has nonzero linear coefficient, it has a
compositional inverse over $\Q$, and $t_\psi=\widehat\nu^{-1}(t_\varphi)$.
Substituting into $x(t)\in t^{-2}\Z[[t]]$ and $y(t)\in t^{-3}\Z[[t]]$ gives
the last assertion.
\end{proof}

The following argument is from \cite[proof of Prop.~2]{Edixhoven1991} and \cite[proof of Thm.~1.6]{Stevens1989}; see also \cite[Thm.~7.2]{Yin2026}.
\begin{lemma}\label{lem:integral-param}
Let $V$ be $X_0(N)$ or $\Xmu$, let $E$ be an elliptic curve over $\Q$, and let
$\varphi:V\to E$ be a morphism over $\Q$ with $\varphi(\infty)=0$. Then
$t_\varphi\in q\,\Z[[q]]$.
\end{lemma}

\begin{proof}
We may assume that $\varphi$ is not constant. Fix a prime $p$. If
$V=X_0(N)$, let $U$ be the smooth locus of $X_0(N)_{\Z_p}$, where
$X_0(N)_\Z$ is the modular curve over $\Z$ of
\cite[Section~1.5]{Cesnavicius2018}. The cusp $\infty$ extends to a section
of $U$ \cite[Rem.~5.3]{Cesnavicius2018}, and the completion of $U$ along it is
$\Spf\Z_p[[q]]$, where $q$ is the analytic parameter at $\infty$
\cite[proof of Prop.~2]{Edixhoven1991}. If $V=\Xmu$, let
$U=\mathcal X_\mu(N)_{\Z_p}$, where $\mathcal X_\mu(N)$ is the coarse moduli
scheme over $\Z$ of generalized elliptic curves with irreducible geometric
fibres equipped with a closed immersion of $\mu_N$ into the smooth locus. By
\cite[proof of Lemma~6.1.6]{CES}, $\mathcal X_\mu(N)$ is smooth over $\Z$, the
cusp $\infty$ extends to a $\Z$-point of it, and the completion along this
point is $\Spf\Z[[q]]$, where $q$ is the analytic parameter at $\infty$. Its
generic fibre is the open subscheme of $\Xmu$ obtained by removing the cusps
other than the $\diam\delta\infty$.

In both cases the N\'eron mapping property
\cite[Section~1.2, Def.~1]{BLR1990} extends $\varphi$ to a morphism
$U\to\cE$. The section $\infty$ of $U$ maps to a section of $\cE$
which is zero on the generic fibre, hence is the zero section, as $\cE$ is
separated. Therefore the composite $\Spf\Z_p[[q]]\to U\to\cE$ factors through
the completion $\Spf\Z_p[[t]]$ of $\cE$ along its zero section, and
$t_\varphi\in q\,\Z_p[[q]]$. The coefficients of $t_\varphi$ are rational
because $\varphi$, the cusp and the parameter $q$ are defined over $\Q$. As
$p$ was arbitrary, $t_\varphi\in q\,\Z[[q]]$.
\end{proof}

If $t_\psi\in q+q^2\Z[[q]]$ for a map $\psi$ to an elliptic curve with
global minimal coordinates $x,y$, then $x\circ\psi$ and $y\circ\psi$ lie in
$\Z((q))$, because $x(t)\in t^{-2}+t^{-1}\Z[[t]]$ and
$y(t)\in-t^{-3}+t^{-2}\Z[[t]]$. We write $\ell_f(q)=\sum_{n\geq1}(a_n/n)q^n$,
so that $\omega_f=\dd\ell_f$.

\begin{proposition}\label{prop:integral-root}
Let $V$ be $X_0(N)$ or $\Xmu$, and let $E/\Q$ be an elliptic curve of
conductor $N$ with newform $f$. Let $\varphi:V\to E$ be a morphism over $\Q$
with $\varphi(\infty)=0$ and
\[
  \varphi^*\omega_E=c\,\omega_f
\]
for an integer $c\geq1$ such that $p^2\mid N$ for every prime $p\mid c$, and
let $\psi$ be the normalized lift of $\varphi$ through $[c]$. Then
\[
  t_\psi=L_E^{-1}(\ell_f)\in q+q^2\Z[[q]],
\]
and $x\circ\psi$, $y\circ\psi$ lie in $\Z((q))$.
\end{proposition}

\begin{proof}
Integrating $\varphi^*\omega_E=c\,\omega_f$ gives $L_E(t_\varphi)=c\,\ell_f$.
Since $L_E\circ[c]_{\widehat E}=c\,L_E$ and $[c]_{\widehat E}(t_\psi)=t_\varphi$
(Lemma~\ref{lem:lift-expansion}), we get $L_E(t_\psi)=\ell_f$, that is,
$t_\psi=L_E^{-1}(\ell_f)$. Its linear coefficient is $a_1=1$, and its
coefficients are rational. It remains to show that they lie in $\Z_p$ for
every prime $p$.

If $p\nmid c$, then $g(w)=c^{-1}[c]_{\widehat E}(w)$ lies in
$w+w^2\Z_p[[w]]$, so $g^{-1}$ has coefficients in $\Z_p$ by
Lemma~\ref{lem:inverse}. Since $g(t_\psi)=c^{-1}t_\varphi$, and
$t_\varphi\in q\,\Z[[q]]$ by Lemma~\ref{lem:integral-param}, we get
$t_\psi=g^{-1}(c^{-1}t_\varphi)\in q\,\Z_p[[q]]$.

If $p\mid c$, then $p^2\mid N$, so $E$ has additive reduction at $p$ and
$a_p=0$. Hence $a_n=0$ whenever $p\mid n$, and every coefficient of $\ell_f$
lies in $\Z_p$. By Lemma~\ref{lem:additive}, $t_\psi=L_E^{-1}(\ell_f)$ lies in
$q\,\Z_p[[q]]$.
\end{proof}

\begin{remark}\label{rem:linear-term}
The parameter of $\varphi$ itself is $t_\varphi=c\,q+O(q^2)$, so $x\circ\varphi$
and $y\circ\varphi$ have denominators at the primes dividing $c$, while
$x\circ\psi$ and $y\circ\psi$ do not. The integrality of $L_E^{-1}(\ell_f)$
also follows from Honda's comparison of formal groups: the formal group over
$\Q$ with logarithm $\sum_{n\geq1}(a_n/n)w^n$ is strictly isomorphic over $\Z$
to $\widehat E$, at every prime and whatever the reduction type of $E$ there
\cite[Thm.~9]{Honda1970}, \cite[Section~3]{Honda1973}. In particular, the conclusion of Proposition~\ref{prop:integral-root} holds without the assumption $p^2\mid N$ for $p\mid c$. We keep this assumption because it allows the direct proof above, and \eqref{eq:support} is also used in Appendix~\ref{app:another} to obtain additive reduction at $2$.
\end{remark}

\begin{proposition}\label{prop:integral-etale}
Let $V$ be $X_0(N)$ or $\Xmu$, and let $E/\Q$ be an elliptic curve of
conductor $N$ with newform $f$. Let $\varphi:V\to E$ be a morphism over $\Q$
with $\varphi(\infty)=0$ and $\varphi^*\omega_E=\omega_f$, let
$\nu:E'\to E$ be an isogeny over $\Q$ with $\nu^*\omega_E=\omega_{E'}$, and
let $\psi$ be the normalized lift of $\varphi$ through $\nu$. 

Then $t_\psi\in q+q^2\Z[[q]]$, and $x\circ\psi$, $y\circ\psi$ lie in $\Z((q))$,
where $x,y$ are the coordinates of a global minimal Weierstrass equation of
$E'$.
\end{proposition}

\begin{proof}
For every prime $p$, $\nu$ extends to a morphism of N\'eron models over $\Z_p$
preserving the zero sections, hence to a homomorphism of their formal groups.
Thus $\widehat\nu\in w\,\Z_p[[w]]$ for all $p$, and its linear coefficient is
$1$ because $\nu^*\omega_E=\omega_{E'}$. Hence
$\widehat\nu\in w+w^2\Z[[w]]$, and $\widehat\nu^{-1}$ has integral
coefficients by Lemma~\ref{lem:inverse}. By Lemma~\ref{lem:integral-param},
$t_\varphi\in q\,\Z[[q]]$, and $L_E(t_\varphi)=\ell_f$ shows that its linear
coefficient is $1$. By Lemma~\ref{lem:lift-expansion},
$t_\psi=\widehat\nu^{-1}(t_\varphi)\in q+q^2\Z[[q]]$.
\end{proof}

% !TEX root = ../paper.tex
\section{Modularity and rationality of normalized lifts}\label{sec:lifting}

In this section we prove Theorem~\ref{thm:lifting}: part~\ref{lift:congruence}
(Modularity) is Proposition~\ref{prop:congruence} together with
Lemma~\ref{lem:monodromy} (Section~\ref{subsec:congruence}), part~\ref{lift:modular}
(Level) is Proposition~\ref{prop:level} (Section~\ref{subsec:modularity}), and
part~\ref{lift:rational} (Rationality) is Propositions~\ref{prop:descent}
and~\ref{prop:translations} (Section~\ref{subsec:rationality}).

We keep the notation of
Section~\ref{subsec:normalized-lift}: $\psi$
is the normalized lift of $\varphi:X_\Gamma\to E$ through $\nu:E'\to E$,
$\rho:\Gamma\to\ker\nu$ is the homomorphism of Lemma~\ref{lem:rho}, and
$G_\psi=\ker\rho$. The argument follows Yin's congruence argument for cubic
roots of modular functions \cite[Prop.~8.1]{Yin2026} in an opposite direction.

For a positive integer $m$ we write
\[
  \Gamma(m)=\ker\bigl(\SL\to\mathrm{SL}_2(\Z/m\Z)\bigr),
\]
and $T=\sm{1&1}{0&1}$.

\subsection{Modularity}\label{subsec:congruence}

Let $X=X_\Gamma(\C)$, a compact Riemann surface, and let
\[
  Y=X\times_{\varphi,E(\C),\nu}E'(\C).
\]
Since $\nu$ is a finite covering, so is $Y\to X$. Thus $Y$ is a compact,
possibly disconnected, Riemann surface, unramified over $X$ at every point,
including the cusps and the images of the elliptic points. Let $Y_\psi$ be
the connected component of $Y$ containing $(\infty,0)$.

\begin{lemma}\label{lem:monodromy}
The homomorphism $\rho$ vanishes on $-I$ and on every parabolic and elliptic
element of $\Gamma$. The map $z\mapsto(z,\psi(z))$ induces an isomorphism of
compact Riemann surfaces $X_{G_\psi}(\C)=G_\psi\backslash\Hup^*\simeq Y_\psi$.
\end{lemma}

\begin{proof}
The element $-I$ acts trivially on $\Hup$, so $\rho(-I)=0$. If $\gamma$ is
elliptic with fixed point $z_0\in\Hup$, then
$\rho(\gamma)=\psi(z_0)-\psi(z_0)=0$. If $\gamma$ is parabolic, it fixes a
cusp $\kappa=\sigma\infty$ with $\sigma\in\SL$, and
$\gamma=\pm\sigma T^r\sigma^{-1}$ for some $r\neq0$. For $\operatorname{Im}z'$
large, the image in $X$ of the path from $\sigma z'$ to
$\sigma(z'+r)=\gamma\sigma z'$ is a loop in a small disc around the image of
$\kappa$, hence null-homotopic in $X$. Since $Y\to X$ is a covering, its lift
$z\mapsto(z,\psi(z))$ is a closed loop, so $\rho(\gamma)=0$.

If $(z,\psi(z))$ and $(z',\psi(z'))$ have the same image in $Y$, then
$z'=\gamma z$ with $\gamma\in\Gamma$ and $\rho(\gamma)=0$, so $\gamma\in
G_\psi$. Let $Y^\circ\subseteq Y$ be the preimage of $\Gamma\backslash\Hup$.
Since every elliptic element of $\Gamma$ lies in $G_\psi$, the finite map
$G_\psi\backslash\Hup\to\Gamma\backslash\Hup$ is unramified, so
$G_\psi\backslash\Hup\to Y^\circ$ is an injective local isomorphism which is
proper, that is, an isomorphism onto a union of connected components of
$Y^\circ$. Its image is connected and contains the points $(z,\psi(z))$ with
$\operatorname{Im}z$ large, which are close to $(\infty,0)$. Hence it is
$Y_\psi\cap Y^\circ$, the complement in $Y_\psi$ of finitely many points over
the cusps. Thus $X_{G_\psi}(\C)$ and $Y_\psi$ are compact Riemann surfaces,
finite over $X$, containing isomorphic complements of finite sets compatibly
with the maps to $X$, and the isomorphism extends to them.
\end{proof}

From now on we assume \BDC{}. We use the unbounded denominators
theorem in the following form \cite[Thm.~1.0.1]{CDT}.

\begin{theorem}[Calegari--Dimitrov--Tang]\label{thm:cdt}
Let $F:\Hup\to\C$ be a modular form of some integral weight for a subgroup of
finite index in $\SL$, holomorphic on $\Hup$ and at the cusps, whose Fourier
expansion at $\infty$ lies in $\Z[[q]]$. Then $F$ is a modular form for a
congruence subgroup.
\end{theorem}

The theorem allows expansions in fractional powers of $e^{\pi iz}$ but we only
need $\Z[[q]]$.

Let $j$ be the modular invariant and
$\Delta=q\prod_{n\geq1}(1-q^n)^{24}$ the cusp form of weight $12$ for $\SL$,
which has no zeros on $\Hup$.

\begin{proposition}\label{prop:congruence}
$G_\psi$ is a congruence subgroup.
\end{proposition}

\begin{proof}
By Lemma~\ref{lem:rho}, $G_\psi$ has finite index in $\SL$. Put
$u_x=x\circ\psi$ and $u_y=y\circ\psi$. By Lemma~\ref{lem:monodromy} they are
meromorphic functions on $X_{G_\psi}(\C)$, and in particular meromorphic at
the cusps. A pole of $u_x$ or $u_y$ in $\Hup$ is a point $z$
with $\psi(z)=0$, hence with $\varphi(z)=0$, so its image in $X$ lies in the
finite set $\varphi^{-1}(0)$. These points are algebraic, because $\varphi$
is defined over $\Q$, and so are their $j$-invariants. The orders of the
poles are bounded, because there are finitely many poles modulo $G_\psi$. Choose a
nonzero polynomial $P$ with integer coefficients vanishing to sufficiently high order at all
these $j$-invariants. Then $P(j)u_x$ and $P(j)u_y$ are holomorphic on $\Hup$,
and for $s$ sufficiently large the functions
\[
  F_x=k\,\Delta^sP(j)\,u_x,\qquad F_y=k\,\Delta^sP(j)\,u_y
\]
are also holomorphic at every cusp, because $\Delta$ vanishes at every cusp.
They are modular forms of weight $12s$ for $G_\psi$, and since
$\Delta\in q\,\Z[[q]]$ and $j\in q^{-1}\Z[[q]]$, \BDC{} shows that
their expansions at $\infty$ lie in $\Z[[q]]$ for sufficiently large $s$.

By Theorem~\ref{thm:cdt} of Calegari-Dimitrov-Tang, $F_x$ and $F_y$ are modular for congruence
subgroups $C_x$ and $C_y$. The factor $k\Delta^sP(j)$ is modular of weight
$12s$ for $\SL$, so $u_x$ and $u_y$ are invariant under $C_x\cap C_y$. The
coordinates $x,y$ separate the points of $E'(\C)\smallsetminus\{0\}$, so for
$\gamma\in C_x\cap C_y$ we have $\psi\circ\gamma=\psi$ outside a discrete
set, hence everywhere. Choose $M_1\geq1$ with $\Gamma(M_1)\subseteq C_x\cap
C_y$. Since $\Gamma(N)\subseteq\Gamma$,
\[
  \Gamma(\lcm(M_1,N))\subseteq C_x\cap C_y\cap\Gamma\subseteq G_\psi .\qedhere
\]
\end{proof}

\begin{proof}[Proof of Theorem~\textup{\ref{thm:lifting}\ref{lift:congruence}}]
By Lemma~\ref{lem:rho}, $G_\psi$ is normal in $\Gamma$. It is a congruence
subgroup by Proposition~\ref{prop:congruence}, and $X_{G_\psi}(\C)\simeq Y_\psi$
by Lemma~\ref{lem:monodromy}.
\end{proof}

\subsection{Level}\label{subsec:modularity}

We first compare cusp widths, using Lemma~\ref{lem:monodromy}. Let
$G\subseteq\SL$ be a subgroup of finite index containing $-I$. If
$\sigma\in\SL$ maps $\infty$ to a cusp $\kappa$, the \emph{width} of $\kappa$
for $G$ is
\[
  w_G(\kappa)=\min\{w\in\Z_{\geq1}:\sigma T^w\sigma^{-1}\in G\},
\]
and the \emph{level} $L(G)$ of $G$ is the least common multiple of its cusp
widths \cite[Def.~4.1.2]{CDT}.

\begin{corollary}\label{cor:widths}
Every cusp has the same width for $G_\psi$ as for $\Gamma$. In particular
$L(G_\psi)=L(\Gamma)$.
\end{corollary}

\begin{proof}
Let $\kappa=\sigma\infty$ be a cusp, with $\sigma\in\SL$, and let $w$ be the
width of $\kappa$ for $\Gamma$, so that $\gamma=\sigma T^{w}\sigma^{-1}\in\Gamma$.
Since $G_\psi\subseteq\Gamma$, the width for $\Gamma$
divides the width for $G_\psi$. Conversely, $\gamma$ is a parabolic element
of $\Gamma$, so $\rho(\gamma)=0$ by Lemma~\ref{lem:monodromy}, that is,
$\gamma\in G_\psi$. Hence the width for $G_\psi$ divides $w$.
\end{proof}

To determine the level we use Wohlfahrt's theorem \cite[Thm.~2]{Wohlfahrt1964},
in the form given at the end of the proof of \cite[Lemma~4.4.1]{CDT}: a
congruence subgroup $G$ containing $-I$ satisfies
\[
  \pm\Gamma(L(G))\subseteq G .
\]

\begin{proposition}\label{prop:level}
$\Gpm\subseteq G_\psi\subseteq\Gamma$.
\end{proposition}

\begin{proof}
Every group between $\Gpm$ and $\Gamma_0(N)$ contains $\Gamma(N)$, so its cusp
widths divide $N$. The cusp $0$ has width exactly $N$ for such a group, since
$\sm{1&0}{-w&1}\in\Gamma_0(N)$ if and only if $N\mid w$. Thus $L(\Gamma)=N$,
and $L(G_\psi)=N$ by Corollary~\ref{cor:widths}. Since $G_\psi$ is a
congruence subgroup containing $-I$ (Lemma~\ref{lem:monodromy} and
Proposition~\ref{prop:congruence}), Wohlfahrt's theorem gives
$\pm\Gamma(N)\subseteq G_\psi$. The parabolic element $T$ lies in $\Gamma$,
hence in $G_\psi$. Reduction modulo $N$ identifies $\Gamma_1(N)/\Gamma(N)$
with $\{\sm{1&b}{0&1}:b\in\Z/N\Z\}$, which is generated by the image of $T$.
Hence $\Gpm=\langle\pm\Gamma(N),T\rangle\subseteq G_\psi$.
\end{proof}

\begin{proof}[Proof of Theorem~\textup{\ref{thm:lifting}\ref{lift:modular}}]
This is Proposition~\ref{prop:level} together with Lemma~\ref{lem:monodromy}.
\end{proof}

\subsection{Rationality}\label{subsec:rationality}

We will first show that the induced map from $\psi$ is algebraic by GAGA. Because it is normalized at the
rational cusp $\infty$, its graph is stable under $\operatorname{Gal}(\Qbar/\Q)$,
and it descends to $\Q$. Its translations under the diamond operators are then
values at rational cusps.

\begin{proposition}\label{prop:descent}
There is a unique morphism $\psi_\mu:\Xmu\to E'$ over $\Q$ with
$\psi_\mu(\infty)=0$ and $\nu\circ\psi_\mu=\varphi\circ\varpi_\Gamma$. Its
analytification is induced by $\psi$.
\end{proposition}

\begin{proof}
By Proposition~\ref{prop:level}, $\psi$ is invariant under $\Gpm$, and by
Lemma~\ref{lem:monodromy} it defines a holomorphic map
$\Xmu(\C)\to X_{G_\psi}(\C)\simeq Y_\psi\to E'(\C)$, still denoted by $\psi$.
By GAGA this map is algebraic. It satisfies
$\nu\circ\psi=\varphi\circ\varpi_\Gamma$ and $\psi(\infty)=0$.

The scheme
\[
  \mathcal Z=\Xmu\times_{\varphi\varpi_\Gamma,E,\nu}E'
\]
is finite \'etale over $\Xmu$ and defined over $\Q$, and
$(\infty,0)\in\mathcal Z(\Q)$. As $\mathcal Z_{\Qbar}$ is a smooth curve over an
algebraically closed field, its connected components are geometrically
connected, so the connected components of $\mathcal Z_\C$ are the base changes
of those of $\mathcal Z_{\Qbar}$. The graph of $\psi$ is closed, and open
because $\mathcal Z\to\Xmu$ is \'etale and the graph maps isomorphically to
$\Xmu$. Thus the graph is a connected component of $\mathcal Z_\C$, and therefore comes
from a connected component of $\mathcal Z_{\Qbar}$ containing $(\infty,0)$.
Distinct connected components are disjoint and $\operatorname{Gal}(\Qbar/\Q)$
fixes $(\infty,0)$, so this component is Galois stable. It descends to a
closed subscheme of $\mathcal Z$ mapping isomorphically to $\Xmu$, and hence
$\psi$ descends to a morphism $\psi_\mu$ over $\Q$.\footnote{Richard Taylor
suggested the following alternative argument using $q$-expansions. The
completion of $\Xmu$ along $\infty$ is $\Spf\Q[[q]]$, so expansion at $\infty$
gives an injective map $\C(\Xmu)\to\C((q))$ which commutes with the action of
$\operatorname{Aut}(\C/\Q)$ on coefficients. By
Lemma~\ref{lem:lift-expansion}, the expansions of $x\circ\psi$ and
$y\circ\psi$ lie in $\Q((q))$, so these functions are fixed by
$\operatorname{Aut}(\C/\Q)$ and lie in $\Q(\Xmu)$. Hence $\psi$ is defined
over $\Q$.}

Two morphisms with these properties differ by a morphism $\Xmu\to\ker\nu$
which vanishes at $\infty$. Since $\Xmu$ is geometrically connected and
$\ker\nu$ is finite, this morphism is zero.
\end{proof}

\begin{proposition}\label{prop:translations}
Let $\gamma=\sm{a&b}{Nr&d}\in\Gamma$. Then
\[
  \psi_\mu\circ\diam d=\psi_\mu+\rho(\gamma),\qquad
  \rho(\gamma)=\psi_\mu(\diam d\infty)\in E'(\Q).
\]
\end{proposition}

\begin{proof}
The class of $d$ lies in $\Delta_\Gamma$, so
$\varpi_\Gamma\circ\diam d=\varpi_\Gamma$, and
$\psi_\mu\circ\diam d-\psi_\mu$ is a morphism from $\Xmu$ to $\ker\nu$. As
$\Xmu$ is geometrically connected, it is constant, equal to its value
$\psi_\mu(\diam d\infty)$ at $\infty$. This value is a rational point because
$\psi_\mu$, $\diam d$ and $\infty$ are defined over $\Q$. Since $\gamma$ acts
on $\Xmu(\C)$ as $\diam d$ (Section~\ref{subsec:normalized-lift}), this constant is
$\psi(\gamma z)-\psi(z)=\rho(\gamma)$.
\end{proof}

\begin{proof}[Proof of Theorem~\textup{\ref{thm:lifting}\ref{lift:rational}}]
The first assertions are Propositions~\ref{prop:descent}
and~\ref{prop:translations}. Since $\diam{-1}$ is the identity,
$\tau_{-d}=\tau_d$, so $H_\psi=\rho(\Gamma)=\{\tau_\delta:\delta\in\Delta_\Gamma\}$.
The component $Y_\psi$ is the image of $\Xmu(\C)$ under
$(\varpi_\Gamma,\psi_\mu)$, a morphism over $\Q$, so it descends to a connected
component $Y_\psi^{\Q}$ of the finite \'etale $X_\Gamma$-scheme
$X_\Gamma\times_{\varphi,E,\nu}E'$. Over $\C$, the subgroup of $\ker\nu$
preserving $Y_\psi$ is $H_\psi$, which acts simply transitively on the fibres
of $Y_\psi\to X$. Therefore $Y_\psi^{\Q}\to X_\Gamma$ is Galois, and its automorphisms
over $X_\Gamma$ are the translations by $H_\psi$. These points are rational, so
the automorphism group scheme is the constant group $H_\psi$. Finally $\rho$
kills $\Gpm$ and the parabolic and elliptic elements of $\Gamma$ by
Lemma~\ref{lem:monodromy} and Proposition~\ref{prop:level}, so
$H_\psi=\rho(\Gamma)$ is a quotient of $\Gamma/\Gamma'$.
\end{proof}

\begin{remark}\label{rem:dimension}
The assumption that $E$ and $E'$ are elliptic curves is used in
Proposition~\ref{prop:congruence}, where $x,y$ separate the points of
$E'(\C)\smallsetminus\{0\}$, and in the local arguments of
Lemmas~\ref{lem:additive} and~\ref{lem:torsion}. For an abelian variety one
could instead impose bounded denominators on the compositions of the lift with
finitely many rational functions separating the points of an open subset met
by the lift. The natural analogues of the Manin constant conjecture for
optimal quotients of higher dimension fail in general
\cite[Thm.~5.10]{Cesnavicius2018}.
\end{remark}

% !TEX root = ../paper.tex
\begingroup
\setlength{\emergencystretch}{2em}
\section{Igusa curves and the local torsion}
\label{sec:c-two}
Let $E/\Q$, $N$, and $c$ be as in Section~\ref{sec:deduction}. By
Proposition~\ref{prop:rational-torsion}, there are a morphism
$\psi_\mu:\Xmu\to E$ over $\Q$ and rational points $\tau_\delta\in E(\Q)$
such that
\begin{equation}\label{eq:inputs}
 \psi_\mu\circ\diam{\delta}=\psi_\mu+\tau_\delta,
 \qquad \{\tau_\delta:\delta\in(\Z/N\Z)^\times\}=E[c].
\end{equation}
Here $c$ is a positive integer. 

In this section we prove
Proposition~\ref{prop:c-one}, which says that no prime divides $c$. This is stronger than the remaining $c\neq 2$. In Appendix~\ref{app:another}, we provide an alternative proof that $c\neq 2$.

Suppose that a prime $\ell$ divides $c$, and choose a prime $p$ at which $E$
has bad reduction, which exists by a theorem of Tate, see \cite[Exercise~8.15]{Silverman2009}. The main point is that
every $\tau_\delta$ lies in $E^0(\Q_p)$, the group of points reducing into
the identity component of the N\'eron model
(Proposition~\ref{prop:identity}). Since $E$ has bad reduction at $p$, the
group $E^0(\Q_p)[\ell]$ has order at most $\ell$
(Lemma~\ref{lem:torsion}). On the other hand $E[\ell]\subseteq E[c]$
consists of $\ell^2$ points of $E^0(\Q_p)$, which is a contradiction.

To prove Proposition~\ref{prop:identity}, we use an integral model of
$\Ymu$ at $p$, where $\Ymu\subset\Xmu$ is the complement of the cusps
(Definition~\ref{def:Xmu}). As we explain in
Section~\ref{subsec:moduli-objects}, the moduli problem defining $\Ymu$ is
represented over $\Z$ by a smooth curve $\mathcal Y_\mu(N)$ with $\mathcal Y_\mu(N)_\Q=\Ymu$.
For every prime $p\mid N$ it has the following properties.
\begin{enumerate}[label=\textup{(\roman*)}]
\item $\mathcal Y_\mu(N)_{\Z_p}$ is smooth over $\Z_p$, and every diamond operator
  $\diam\delta$ extends to a $\Z_p$-automorphism of $\mathcal Y_\mu(N)_{\Z_p}$.
\item The geometric special fibre of $\mathcal Y_\mu(N)_{\Z_p}$ is nonempty and
  connected (Theorem~\ref{thm:geometry}).\footnote{The connectedness comes
  from Igusa's theorem that the monodromy of the Igusa covering of the
  ordinary locus is surjective \cite[Thm.~4.3]{Katz1973}. Its proof shows that
  the covering is already totally ramified at every supersingular point
  \cite[Thm.~4.3 bis]{Katz1973}. We follow instead the more modern treatment
  of Igusa curves in \cite[Chapter~12]{KM1985}.}
\end{enumerate}
By (i) and the N\'eron mapping property, $\psi_\mu$ extends to a morphism $h$
from $\mathcal Y_\mu(N)_{\Z_p}$ to the N\'eron model of $E$ over $\Z_p$, and
$h\circ\diam\delta=h+\tilde\tau_\delta$, where $\tilde\tau_\delta$ is the
section of the N\'eron model extending $\tau_\delta$. By (ii), $h$ maps the
geometric special fibre of $\mathcal Y_\mu(N)_{\Z_p}$ into one component of the
special fibre of the N\'eron model, so $\tau_\delta$ reduces into the
identity component.

The proof in Section~\ref{sec:deduction} that $c\mid2$ uses only the
rationality of the points $\tau_\delta$. The argument of this section also
uses the geometry of the special fibre of $\mathcal Y_\mu(N)_{\Z_p}$. We expect these geometric arguments to be useful for bounding the Manin ideal of elliptic curves over general number fields, where analogues of the case $c=2$ occur more often, and for other problems related to
Theorem~\ref{thm:lifting}.

\subsection{The moduli problems}
\label{subsec:moduli-objects}

Fix $p\mid N$ and write $N=p^aM$, where $a\geq1$ and $p\nmid M$. Let
$k=\overline{\F}_p$.

We use the conventions of \cite[Sections~4.1--4.3 and~4.13]{KM1985}.
For a ring $R$, the objects of $(\mathrm{Ell}/R)$ are elliptic curves $A/W$
over variable $R$-schemes, and its morphisms are cartesian squares
over $R$.
A moduli problem $\mathcal F$ is a contravariant set-valued functor
on this category. Thus $\mathcal F(A/W)$ is a set of additional
structures on the \emph{fixed} elliptic curve $A/W$.
When $\mathcal F$ is representable, we write
$\mathfrak M(\mathcal F)$ for its moduli scheme, following
\cite[(4.3.1)]{KM1985}:
its $W$-points are isomorphism classes of elliptic curves over $W$
equipped with the specified structure.

For a ring $R$ and an integer $m\geq1$, let $\mathcal I_m$ be the moduli
problem on $(\mathrm{Ell}/R)$ which assigns to $A/W$ the set
$\Incl_{W\text{-gp}}(\mu_{m,W},A[m])$ of \cite[Section~4.9]{KM1985}. Its
elements are the f.p.p.f.\ injective homomorphisms, which are exactly the
closed immersions $\iota:\mu_{m,W}\hookrightarrow A[m]$ of $W$-group schemes,
since the group schemes are finite. By \cite[Section~4.9]{KM1985},
$\mathcal I_m$ is relatively representable, and for every $A/W$ the
$W$-scheme $(\mathcal I_m)_{A/W}$ representing it \cite[(4.2)]{KM1985} is
affine and \'etale over $W$.

Since $N\geq4$, the moduli problem $\mathcal I_N$ on $(\mathrm{Ell}/\Z)$ is
rigid \cite[Cor.~2.7.4]{KM1985}. Hence it is represented by a smooth curve
$\mathcal Y_\mu(N)$ over $\Z$ \cite[Cor.~4.7.1]{KM1985}, denoted $Y(\mu_N)$
at the end of \cite[Section~4.10]{KM1985}. The pairs $(A/W,\iota)$ of
Definition~\ref{def:Xmu} are exactly the elliptic curves over $\Q$-schemes
with a level $\mathcal I_N$ structure. Thus $\Ymu=\mathcal Y_\mu(N)_\Q$
\cite[(8.1.1)]{KM1985}, and $\Xmu$ is its compactification. For
$\delta\in(\Z/N\Z)^\times$, the rule
\[
 \diam{\delta}(A,\iota)=(A,\iota\circ[\delta])
\]
is an automorphism of the moduli problem $\mathcal I_N$, hence of $\mathcal Y_\mu(N)$
over $\Z$. Over $\Q$ it is the diamond operator of
Section~\ref{subsec:normalized-lift}.

To prove that the geometric special fibre $\mathcal Y_\mu(N)_k$ is connected,
we use the Igusa curves of \cite[Chapter~12]{KM1985}. Their connectedness
\cite[Cor.~12.6.2]{KM1985} is proved over a representable moduli problem,
whereas $\mathcal I_M$ on $(\mathrm{Ell}/k)$ is not representable when
$M\leq3$. We therefore enlarge the level away from $p$. Choose an integer
$t\geq4$ prime to $N$, and put $L=Mt$. Then $L\geq4$, $p\nmid L$ and
$Nt=Lp^a$. Since $L\geq4$ and $Nt\geq4$, the same references show that the
moduli problems $\mathcal I_L$ and $\mathcal I_{Nt}$ are represented by the
smooth $\Z$-curves
$\mathcal Y_\mu(L)$ and $\mathcal Y_\mu(Nt)$. Moreover $\mathcal I_L$ is
\'etale and surjective over $(\mathrm{Ell}/k)$, since $L$ is invertible in
$k$.

\begin{lemma}\label{lem:base-connected}
The curve $\mathcal Y_\mu(L)_k$ is nonempty and connected.
\end{lemma}

\begin{proof}
A primitive $L$-th root of unity $\zeta_L\in k$ gives an isomorphism
$(\Z/L\Z)_k\simeq\mu_{L,k}$. Since $L$ is invertible in $k$, it identifies
$\mathcal I_L$ on $(\mathrm{Ell}/k)$ with the moduli problem $[\Gamma_1(L)]$
of points of exact order $L$ \cite[Section~3.2]{KM1985}. The compactified
moduli scheme of $[\Gamma_1(L)]$ is a proper smooth curve over $\Z[1/L]$
with geometrically connected fibres
\cite[Cor.~10.9.5(2) and Table~10.9.6]{KM1985}. As $p\nmid L$, the curve
$\mathcal Y_\mu(L)_k$ is the complement of the finitely many cusps in its
geometric fibre at $p$, hence nonempty and connected.
\end{proof}

We now recall Igusa structures. For an elliptic curve $A/W$ over a
$k$-scheme, let $A^{(p^a)}=A\times_{W,F_W^a}W$, where $F_W$ is the absolute
Frobenius of $W$, and let $F^a:A\to A^{(p^a)}$ be the iterated relative
Frobenius and $V^a:A^{(p^a)}\to A$ its dual Verschiebung
\cite[Sections~12.1--12.2]{KM1985}. Let $[\Ig(p^a)]$ be the base change to
$k$ of the Igusa moduli problem \cite[(12.3.1.1)]{KM1985}. It assigns to
$A/W$ the set of points $Z\in A^{(p^a)}(W)$ which generate $\ker V^a$ in the
sense of \cite[1.4.1]{KM1985}, and it is also defined at supersingular
points. On the ordinary locus $\ker V^a$ is \'etale
\cite[Prop.~12.3.6]{KM1985}, so $Z$ is such a generator if and only if
$1\mapsto Z$ is an isomorphism $(\Z/p^a\Z)_W\to\ker V^a$
\cite[(1.5.2), Props.~1.10.6 and~1.10.12]{KM1985}.

\begin{lemma}\label{lem:mu-igusa}
Let $A/W$ be an elliptic curve over a $k$-scheme.
\begin{enumerate}[label=\textup{(\arabic*)}]
\item If $\mathcal I_{p^a}(A/W)$ is nonempty, then every geometric fibre of
  $A/W$ is ordinary.
\item If every geometric fibre of $A/W$ is ordinary, then every
  $\iota\in\mathcal I_{p^a}(A/W)$ induces an isomorphism
  $\iota_0:\mu_{p^a,W}\to\ker F^a$, and $\iota\mapsto(\iota_0^D)^{-1}(1)$ is
  a bijection $\mathcal I_{p^a}(A/W)\to[\Ig(p^a)](A/W)$, compatible with base
  change.
\end{enumerate}
\end{lemma}

\begin{proof}
Assertion (1) concerns geometric fibres, so we may assume that $W$ is the
spectrum of an algebraically closed field. The Weil pairing
identifies $A[p]$ with its Cartier dual, so the restriction of
$\iota\in\mathcal I_{p^a}(A/W)$ to $\mu_p$ dualizes to a surjection
$A[p]\to\Z/p\Z$. Hence $A[p](W)\neq0$, and $A$ is ordinary.

For (2), there is an exact sequence
\[
 0\longrightarrow\ker F^a\longrightarrow A[p^a]
 \xrightarrow{F^a}\ker V^a\longrightarrow0
\]
whose outer terms have rank $p^a$, with $\ker V^a$ \'etale and locally cyclic
\cite[Lemmas~12.2.1, 12.2.3 and 12.2.5, and Prop.~12.3.6]{KM1985}. Every
homomorphism $\mu_{p^a,W}\to\ker V^a$ is zero, since it vanishes on
geometric fibres and its vanishing locus is open and closed
\cite[Lemma~4.10.0(3)]{KM1985}. Hence every $\iota\in\mathcal I_{p^a}(A/W)$
induces an isomorphism $\iota_0$ onto $\ker F^a$, as both have rank $p^a$.
Cartier duality, using the canonical principal polarizations, identifies
$(\ker F^a)^D$ with $\ker V^a$ \cite[(2.8.2.1)]{KM1985}. Thus
$\iota\mapsto(\iota_0^D)^{-1}$ gives
bijections
\[
 \mathcal I_{p^a}(A/W)
 \simeq\Isom(\mu_{p^a,W},\ker F^a)
 \simeq\Isom((\Z/p^a\Z)_W,\ker V^a),
\]
and the right side is $[\Ig(p^a)](A/W)$ via $\alpha\mapsto\alpha(1)$.
\end{proof}

Since $[\Ig(p^a)]$ is relatively representable
\cite[Thm.~12.6.1(1)]{KM1985}, the simultaneous moduli problem
$(\mathcal I_L,[\Ig(p^a)])$ on $(\mathrm{Ell}/k)$ is representable
\cite[(4.3.4)]{KM1985}, and we put
\[
 C_a=\mathfrak M(\mathcal I_L,[\Ig(p^a)]).
\]
The map $C_a\to\mathcal Y_\mu(L)_k$ which forgets $Z$ is finite flat of
degree $p^{a-1}(p-1)$ \cite[Thm.~12.6.1(1)]{KM1985}. It lies over all of
$\mathcal Y_\mu(L)_k$, including the supersingular points.

\subsection{\texorpdfstring{Connectedness of the special fibre}{Connectedness of the special fibre}}
\label{subsec:igusa}

\begin{theorem}\label{thm:geometry}
The geometric special fibre $\mathcal Y_\mu(N)_k$ is nonempty and connected.
\end{theorem}

\begin{proof}
Since $\mu_{Nt}=\mu_L\times\mu_{p^a}$, a level $\mathcal I_{Nt}$ structure is
a pair of level $\mathcal I_L$ and $\mathcal I_{p^a}$ structures. By
Lemma~\ref{lem:mu-igusa}, $\mathcal Y_\mu(Nt)_k$ is therefore the open
subscheme of $C_a$ lying over the ordinary locus of $\mathcal Y_\mu(L)_k$.
By Lemma~\ref{lem:base-connected} and \cite[Cor.~12.6.2]{KM1985}, $C_a$ is
a connected smooth curve, hence irreducible. There are only finitely many
supersingular $j$-values \cite[Lemma~12.5.4]{KM1985}, and $C_a\to
\mathcal Y_\mu(L)_k$ is finite flat of positive degree. Hence
$\mathcal Y_\mu(Nt)_k$ is a nonempty open subscheme of $C_a$, so it is
connected.

Restricting $\iota$ to $\mu_N\subset\mu_{Nt}$ defines a morphism
$r:\mathcal Y_\mu(Nt)_k\to\mathcal Y_\mu(N)_k$. It is surjective on
$k$-points. Indeed, as $t$ is prime to $N$, we have
$\mu_{Nt}=\mu_N\times\mu_t$. Given a $k$-point $(A,\iota_N)$, choose a
closed immersion $\iota_t:\mu_t\hookrightarrow A[t]$. Then
$(x,y)\mapsto\iota_N(x)+\iota_t(y)$ is a level $\mathcal I_{Nt}$ structure
on $A$ whose restriction to $\mu_N$ is $\iota_N$. Therefore
$\mathcal Y_\mu(N)_k$ is nonempty. If it were the disjoint union of two
nonempty open subschemes, each would contain a $k$-point, and their inverse
images under $r$ would be two disjoint nonempty open subschemes covering
$\mathcal Y_\mu(Nt)_k$. This is impossible, so $\mathcal Y_\mu(N)_k$ is
connected.
\end{proof}

\subsection{\texorpdfstring{Proof that $c=1$}{Proof that c=1}}
\label{subsec:local-torsion}
We now apply the geometric result to the translation relation
\eqref{eq:inputs}.

\begin{proposition}\label{prop:identity}
For every $p\mid N$ and every $\delta\in(\Z/N\Z)^\times$,
$\tau_\delta\in E^0(\Q_p)$. Thus
$E[c]\subseteq E^0(\Q_p)$.
\end{proposition}

\begin{proof}
Let $\cE/\Z_p$ be the N\'eron model of $E$. Since $\mathcal Y_\mu(N)_{\Z_p}$
is smooth over $\Z_p$, the N\'eron mapping property
\cite[Section~1.2, Def.~1]{BLR1990} extends the restriction of $\psi_\mu$ to
$\Ymu_{\Q_p}$ uniquely to a morphism $h:\mathcal Y_\mu(N)_{\Z_p}\to\cE$.
The same property extends $\tau_\delta$ to a section
$\tilde\tau_\delta\in\cE(\Z_p)$. The diamond operator $\diam\delta$ is an
automorphism of $\mathcal Y_\mu(N)_{\Z_p}$
(Section~\ref{subsec:moduli-objects}). Then \eqref{eq:inputs} implies, by
uniqueness in the N\'eron mapping property,
$h\circ\diam\delta=h+\tilde\tau_\delta$.
\[
\begin{tikzcd}
  \Ymu_{\Q_p} \arrow[r, hook] \arrow[d, "\psi_\mu"'] & \mathcal Y_\mu(N)_{\Z_p} \arrow[d, dashed, "h"] \\
  E_{\Q_p} \arrow[r, hook] & \cE
\end{tikzcd}
\qquad\qquad
\begin{tikzcd}
  \mathcal Y_\mu(N)_{\Z_p} \arrow[r, "\diam\delta"] \arrow[d, "h"'] & \mathcal Y_\mu(N)_{\Z_p} \arrow[d, "h"] \\
  \cE \arrow[r, "{+\tilde\tau_\delta}"] & \cE
\end{tikzcd}
\]

The image of the connected scheme $\mathcal Y_\mu(N)_k$
(Theorem~\ref{thm:geometry}) lies in one component of $\cE_k$. For a
$k$-point $x$ of $\mathcal Y_\mu(N)_k$, the points $h(x)$ and
$h(\diam\delta x)$ lie in that component. Their difference is
$\tilde\tau_\delta(k)$, whose component class is therefore zero.
Hence the reduction of $\tau_\delta$ lies in the identity component of
$\cE_k$. The identity component of $\cE_{\F_p}$ is geometrically
irreducible \cite[Tag~0B7R]{Stacks}, so its base change to $k$ is the
identity component of $\cE_k$. Therefore
$\tau_\delta\in\cE^0(\Z_p)=E^0(\Q_p)$. The final assertion follows from
\eqref{eq:inputs}.
\end{proof}

\begin{lemma}\label{lem:torsion}
Let $A/\Q_p$ be an elliptic curve with bad reduction, and let $\ell$ be a
prime. Then $\#A^0(\Q_p)[\ell]\leq\ell$.
\end{lemma}

\begin{proof}
The group $A^0(\Q_p)$ consists of the points whose reduction on a minimal
Weierstrass equation is nonsingular
\cite[App.~C, Thm.~15.1(d)]{Silverman2009}. Let $\tilde A_{\mathrm{ns}}$ be
the nonsingular locus of the reduced equation, and let $A^1(\Q_p)$ be the
kernel of reduction. There is an exact sequence
\[
 0\longrightarrow A^1(\Q_p)\longrightarrow A^0(\Q_p)\longrightarrow
 \tilde A_{\mathrm{ns}}(\F_p)\longrightarrow0,
\]
and $A^1(\Q_p)\simeq\widehat A(p\Z_p)$, where $\widehat A$ is the formal
group of the minimal equation
\cite[Props.~VII.2.1 and VII.2.2]{Silverman2009}. Since the reduction is
bad, $\tilde A_{\mathrm{ns}}(\F_p)$ is isomorphic to $\F_p$, to
$\F_p^\times$, or to the group of norm-one elements of $\F_{p^2}^\times$
\cite[Prop.~III.2.5 and Exercise~3.5]{Silverman2009}. In each case it is
cyclic, so its $\ell$-torsion has order at most $\ell$.

If $\ell\neq p$, then $A^1(\Q_p)$ has no point of order $\ell$
\cite[Prop.~VII.3.1(a)]{Silverman2009}. If $\ell=p$ is odd, the formal
logarithm identifies $\widehat A(p\Z_p)$ with the torsion-free group
$p\Z_p$ \cite[Thm.~IV.6.4(b)]{Silverman2009}. In both cases
$A^0(\Q_p)[\ell]$ embeds into $\tilde A_{\mathrm{ns}}(\F_p)[\ell]$.

Let $\ell=p=2$. If the reduction is multiplicative, then
$\tilde A_{\mathrm{ns}}(\F_2)$ has order $1$ or $3$, so
$A^0(\Q_2)[2]=A^1(\Q_2)[2]$. The subgroup $\widehat A(4\Z_2)$ is
torsion-free \cite[Thm.~IV.6.4(b)]{Silverman2009}, and
$\widehat A(2\Z_2)/\widehat A(4\Z_2)\simeq\F_2$
\cite[Prop.~IV.3.2(a)]{Silverman2009}, so $\#A^1(\Q_2)[2]\leq2$.
If the reduction is additive, then $\tilde A_{\mathrm{ns}}\simeq\mathbb G_a$.
The formal group law of $\widehat A$ has coefficients in
$\Z[a_1,\dots,a_6]$, and its reduction modulo $2$ is the formal group law of
the reduced equation, which expands the group law of
$\tilde A_{\mathrm{ns}}$ at the identity \cite[Section~IV.1]{Silverman2009}.
Multiplication by $2$ is zero on $\mathbb G_{a,\F_2}$, hence
\[
 [2]_{\widehat A}(T)=2T+\sum_{n\geq2}b_nT^n,\qquad b_n\in2\Z_2 .
\]
For $0\neq T\in2\Z_2$ and $n\geq2$ we have
$v_2(b_nT^n)\geq1+nv_2(T)>v_2(2T)$, so $[2]_{\widehat A}(T)\neq0$. Thus
$A^1(\Q_2)[2]=0$, and $A^0(\Q_2)[2]$ embeds into
$\tilde A_{\mathrm{ns}}(\F_2)\simeq\F_2$.
\end{proof}

\begin{proposition}\label{prop:c-one}
Let $E$ and $c$ be as in Section~\ref{sec:deduction}. Then no prime
divides $c$, that is, $c=1$. In particular $c\neq2$.
\end{proposition}
\begin{proof}
Suppose that a prime $\ell$ divides $c$. By a theorem of Tate, no elliptic
curve over $\Q$ has good reduction at every prime
\cite[Exercise~8.15 and Example~IX.6.1.1]{Silverman2009}, so $E$ has bad
reduction at some prime $p$, and $p\mid N$.
Equation~\eqref{eq:inputs} gives
$E[\ell]\subseteq E[c]\subseteq E(\Q)$.
Proposition~\ref{prop:identity} and Lemma~\ref{lem:torsion} then imply
\[
 \ell^2=\#E[\ell]\leq\#E^0(\Q_p)[\ell]\leq \ell,
\]
a contradiction. Therefore $c=1$.
\end{proof}

\endgroup

\section*{Declaration of generative AI use}
The proof was initially obtained by ChatGPT-6 Pro. Subsequently, we used Codex
6-Astra, Claude Opus~5.5, and a verification framework OSVerify \cite{osv} with
Codex 5.6-Sol to review the output and improve writing. The authors verified the
output and improved the exposition. The details of the use of AI are further
disclosed in Appendix~\ref{app:ai} for transparency. The authors are responsible
for the correctness of this paper.

\appendix
% !TEX root = ../paper.tex
\section{Details of generative AI usage}\label{app:ai}

This appendix records how generative AI tools were used in producing this
paper, for transparency and honesty.

\subsection{Persistent prompting with math guidance: the initial proof}
Initially, the result was obtained by heavily prompting ChatGPT-6 Pro,
Claude Fable~5.1 and Codex 6-Astra.

\subsubsection{Initial attempt: $c_\varphi\mid\deg\varphi$}
On October 3rd, Ping Nie, who had not previously studied number theory,
prompted ChatGPT-6 Pro and Claude Fable~5.1 to solve the classical Manin
constant conjecture. The agents claimed to have solved the conjecture, but
the authors found that only the weaker statement $c_\varphi\mid\deg\varphi$
had been proved.

\subsubsection{Breakthrough: $c_E\mid2$}
Chi then continued prompting, building on the initial attempt:
\begin{enumerate}
    \item What are the current status and remaining cases? Has the conjecture been reduced to finitely many cases? Is the remaining work ``low-hanging fruit''?
    \item Assuming these results and techniques from Nie's GitHub repo are correct. Prove the actual conjecture.
\end{enumerate}

However, the model remained stuck on the ``Hecke and lattice'' approach, and
Chi kept prompting as issues occurred:
\begin{enumerate}
    \item Try to find a different tech to solve this conjecture because the old techs were not working by the records from the repo.
    \item Try to prove the two comparisons for all conductors.
    \item Consider global number-theoretic properties, such as torsions and rational points, rather than focusing only on local geometry. If we can prove $c_E\in\{1,2\}$, that would be substantial progress.
\end{enumerate}
Prompt (3) became the turning point, and the model produced a proof that
$c_E|2$.

Chi then asked the agent to write a Markdown file for a new agent to audit.
Two independent runs of Codex 6-Astra at \texttt{max} reasoning effort
reviewed the file and reported that it was correct.

\subsubsection{Finally: $c_E=1$}

Chi further pushed the agent by prompting
\begin{enumerate}
    \item So short? Are you sure the argument is detailed enough?
    \item Try to prove that the constant divides 1.
    \item If you cannot find an example, try to prove that none exists.
    \item If the proof requires finiteness of the Tate--Shafarevich group, abandon that approach. Continue trying to prove the full conjecture.
\end{enumerate}
After these rounds of prompts, ChatGPT-6 Pro finally proved $c_E=1$.

\subsection{Non-math prompting with feasibility hint: a second proof}
After Nie was informed that Chi had used AI to generate a 22-page proof but did not see the proof, he restarted a fresh Codex 6-Astra session at \texttt{max} reasoning effort
from the Git state in which only $c_\varphi\mid\deg\varphi$ had been proved,
and prompted:
\begin{quote}
``It was the wrong route and you only proved a weaker result. You should
attack the actual Manin conjecutre. Summarize and learn from the past
failures. It is known that there is a 22-page proof with a novel new idea.''
\end{quote}
The agents then kept working and provided a cleaner argument. Nie asked his
agents to proofread the output carefully and also to search for other
conjectures that could be resolved using the same method. This search led to
the proof of Stevens' conjectures (Theorem~\ref{thm:stevens}). The combined output
became the backbone of the current paper.

\subsection{Agentic AI reviewer}

Subsequently, Chi ran 12 independent checks of the output from Nie's agents
using Codex 6-Astra and Claude Opus~5.5. All 12 checks passed and reported no
mistakes.

However, Liu checked the output with OSVerify \cite{osv} running on
GPT-5.6-Sol at \texttt{xhigh} reasoning effort, which discovered one mistake
(the claim that $L_E\in\Q[[w]]$ for every $E/\Q_p$ is wrong) and suggested a
fix for the three places affected by it. Liu then used OSVerify with
GPT-6-Astra at \texttt{max} reasoning effort and found no further issues.

The authors then used Codex-6-Astra and Claude-Opus-5.5 to improve the writing and to find relevant literature and previous work.

\subsection{Agentic systems: one-shot reproduction of the proof}

Liu used the automated mathematical reasoning agent Rethlas \cite{Rethlas}
(running on GPT-6-Astra at \texttt{max} reasoning effort), enhanced by the
verification framework OSVerify \cite{osv} (running on GPT-5.6-Sol at
\texttt{xhigh} reasoning effort), to attack the Manin constant conjecture.
Only the statement of the conjecture was provided. The system automatically
reproduced essentially the same proof as the one Nie obtained.

\subsection{Where the idea came from}

Probably the most important idea in the proof is to prove modularity of the
normalized lift using \cite{CDT}. When the authors ask Codex-6-Astra to review the literature, list all previous study and check novelty, the agents reported \cite[Proposition 8.1]{Yin2026} as ``the closest earlier use of the method''.

However, as Nie and Liu independently investigated the tool calls and websearch histroy, they both found that the agents did not search or read \cite{Yin2026} before obtaining the solution. One of the miracle moment was that the agent decided to search ``\texttt{Manin+unbounded denominators}''.

Liu also used the mathematics reasoning agent \cite{Rethlas} and reproduced the proof. This agent was, by design, required to download everything that it read from the Internet. Liu used Claude Code Fable-5.1 to monitor the process and report every websearch and download attempt. The memory and log files of Rethlas, together with the report by Claude Code, did not mention \cite{Yin2026}.

This may be indications that AI could come up with this novel idea independently. However,
\begin{enumerate}
    \item The models may have learnt \cite{Yin2026} before the inference time. There is no solid evidence to exclude this possibilty.
    \item Yiyang Wang independently prompted Codex-6-Astra to ask questions about \cite{Yin2026}, and then a single prompt ``/goal prove manin constant c|2'' reproduced the proof quickly. This indicates Yin's method is an extremely strong hint to the model.
\end{enumerate}

% !TEX root = ../paper.tex
\begingroup
\setlength{\emergencystretch}{2em}
\section{\texorpdfstring{Another proof that $c\neq2$}{Another proof that c is not 2}}\label{app:another}

Let $E$, $\pi_0$, $\varphi_0$ and $c$ be as in Section~\ref{sec:deduction}. In
this appendix, we give another proof that $c\neq2$
(Theorem~\ref{thm:app}). Together with $c\mid2$, which follows from the Weil
pairing (see the proof of Theorem~\ref{thm:manin} in
Section~\ref{subsec:proof-manin}), it gives $c=1$, and hence a proof of
Theorem~\ref{thm:manin} that does not use Section~\ref{sec:c-two}.

This alternative proof\footnote{It is a polished version of the first proof obtained with ChatGPT-6 Pro (see Appendix~\ref{app:ai}).} is specific to
excluding $c=2$ over $\Q$ and crucially uses \eqref{eq:support} deduced from \cite[Thm.~2.11]{Cesnavicius2018}. We argue by contradiction, assuming $c=2$. Its key step, Proposition~\ref{prop:app-torsion}, is an
explicit computation at $p=2$, which is both the prime dividing $c$ and a
prime of additive reduction: it uses that the points of order $2$ are given
by the roots of a cubic. The proof of Theorem~\ref{thm:app} also
uses that $\Z_2^\times$ modulo $\pm1$ and the squares has order $2$. For an
odd prime $\ell\mid c$, the corresponding computation takes place over local
fields of residue characteristic $\ell$ containing $\mu_\ell$, which are
ramified over $\Q_\ell$, while over other number fields, it takes place over finite
extensions of $\Q_\ell$, which may be ramified. In either case none of these
facts has a direct analogue. For this reason, the main text uses instead the
argument of Section~\ref{sec:c-two}, which applies to every prime $\ell\mid c$
at any prime $p$ of bad reduction.

Until the end of this appendix, we assume that
$c=2$ and expect to obtain a contradiction. We fix the following notation. By \eqref{eq:support} we have
$4\mid N$, so every elliptic curve over $\Q$ isogenous to $E$ has additive
reduction at $2$ (including $E$ itself). Let $\psi_\mu:\Xmu\to E$ be the morphism induced by the
normalized lift of $\varphi_0$ through $[2]$, and let
$\tau_\delta=\psi_\mu(\diam\delta\infty)\in E(\Q)$, for
$\delta\in(\Z/N\Z)^\times$, be as in Proposition~\ref{prop:rational-torsion},
so that
\[
 [2]\circ\psi_\mu=\varphi_0\circ\varpi,\qquad
 \psi_\mu\circ\diam\delta=\psi_\mu+\tau_\delta,\qquad
 \{\tau_\delta:\delta\in(\Z/N\Z)^\times\}=E[2],
\]
where $\varpi:\Xmu\to X_0(N)=\Xmu/\Delta_{\Gamma_0(N)}$,
$(A,\iota)\mapsto(A,\iota(\mu_N))$, is the quotient map
(see Section~\ref{subsec:normalized-lift}).
Since $\diam{\delta\delta'}=\diam\delta\diam{\delta'}$ and $\diam{-1}$ is the
identity, the map
\[
 \tau:(\Z/N\Z)^\times\longrightarrow E[2],\qquad\delta\longmapsto\tau_\delta,
\]
is a surjective homomorphism with $\tau(-1)=0$. Write $N=2^hM$ with $M$ odd;
then $h\geq2$. By the Chinese remainder theorem, we identify
$(\Z/M\Z)^\times$ with the subgroup of $(\Z/N\Z)^\times$ of classes congruent
to $1$ modulo $2^h$, and we denote by $\tau_M$ the restriction of $\tau$ to
it. Let $K=\Q(\mu_M)$ be the cyclotomic field generated by the $M$-th roots
of unity; it is unramified at $2$ since $M$ is odd.

Our main ingredients are the Galois action on the values of $\psi_\mu$ at cusps of $\Xmu$, and Theorem~\ref{thm:lifting} applied to a curve $2$-isogenous to $E$. Roughly speaking, the factor $(\Z/2^h\Z)^\times$ of
$(\Z/N\Z)^\times$ is too small to map onto $E[2]$, so $\tau_M\neq0$. We
construct a cusp $\kappa\in\Xmu(K)$ such that the Galois group moves the
torsion point $P=\psi_\mu(\kappa)$ by the values of $\tau_M$
(Proposition~\ref{prop:app-cusp}). As $\tau_M\neq0$, the $2$-primary part of
$P$ is not rational, so it has order at least $4$. Since $K$ is unramified at
$2$, this gives a point of order $4$ of $E$ over an unramified extension of
$\Q_2$. On the other hand, $E$ has no such point
(Proposition~\ref{prop:app-torsion}) which yields a contradiction. The proof of this fact uses the assumption $c=2$ again, through Theorem~\ref{thm:lifting} applied to a curve
$2$-isogenous to $E$. 

There are different methods to present the proof, but the main idea is the same: write $E$ in a normal form adapted to its points of order $2$, and then show that no point of order $2$ is divisible by $2$ over an unramified extension of $\Q_2$.

\subsection{\texorpdfstring{A cusp over $K=\Q(\mu_M)$}{A cusp over K=Q(mu M)}}\label{subsec:app-cusp}

In this subsection, we construct a cusp $\kappa\in\Xmu(K)$ on which the Galois
group acts through the diamond operators $\diam\delta$ with
$\delta\in(\Z/M\Z)^\times$. The cusp is the
reduction of a point given by a Tate curve. 

Fix a primitive $M$-th root of
unity $\zeta_M\in K$. Let $q$ be an indeterminate, and let
$\mathrm{Tate}(q^M)=\Gm/q^{M\Z}$ be the Tate curve over $K((q))$: for every
finite extension $L$ of $K((q))$ we have $\mathrm{Tate}(q^M)(L)=L^\times/q^{M\Z}$,
functorially in $L$ \cite[Thm.~V.3.1(d) and Remark~V.3.1.2]{Silverman1994}. Using
$\mu_N=\mu_{2^h}\times\mu_M$, the
homomorphism
\[
 \iota:\mu_N\longrightarrow\mathrm{Tate}(q^M)[N],\qquad
 (z,\zeta_M^j)\longmapsto zq^j\quad(z\in\mu_{2^h},\ j\in\Z),
\]
is defined over $K((q))$ and injective, hence a closed immersion of finite
\'etale group schemes. Thus the pair $(\mathrm{Tate}(q^M),\iota)$ defines a
point $x\in\Xmu(K((q)))$. Since $\Xmu$ is proper, $x$ extends to a point of $\Xmu(K[[q]])$. Let
$\kappa\in\Xmu(K)$ be its reduction at $q=0$. Then $\kappa$ is a cusp, since
$j(\mathrm{Tate}(q^M))$ has a pole of order $M$ at $q=0$, where $j$ is the $j$-invariant.\footnote{In
terms of generalized elliptic curves \cite[Section~4.2]{CES}, $\kappa$ is given
by the standard N\'eron $M$-gon over $K$, with $\mu_N$ embedded in its smooth
locus $\Gm\times\Z/M\Z$ by $(z,\zeta_M^j)\mapsto(z,j)$; the image meets every
irreducible component, which is the ampleness condition of
\cite[Def.~4.2.2]{CES}. It is not the $1$-gon: in the Weierstrass model of
$\mathrm{Tate}(q^M)$ over $K[[q]]$, the points $q^j$ with $0<j<M$ reduce to
the node. Compare the cusp $\infty$, which is given by the $1$-gon
(Section~\ref{subsec:normalized-lift}).}

\begin{proposition}\label{prop:app-cusp}
The point $P=\psi_\mu(\kappa)$ is a torsion point of $E(K)$, and
$\sigma(P)=P+\tau_M(t_\sigma)$ for all $\sigma\in\operatorname{Gal}(\Qbar/\Q)$,
where $t_\sigma\in(\Z/M\Z)^\times$ is defined by $\sigma(\zeta)=\zeta^{t_\sigma}$
for all $\zeta\in\mu_M$.
\end{proposition}

\begin{proof}
Let $\sigma\in\operatorname{Gal}(\Qbar/\Q)$. Recall that $(\Z/M\Z)^\times$ is
identified with the subgroup of $(\Z/N\Z)^\times$ of classes congruent to $1$
modulo $2^h$; thus $t_\sigma^{-1}$ acts on $\mu_N=\mu_{2^h}\times\mu_M$ by
$(z,\zeta)\mapsto(z,\zeta^{t_\sigma^{-1}})$. We first show that
$\sigma(\kappa)=\diam{t_\sigma^{-1}}\kappa$. Let
$\sigma$ act on $K((q))$ through the coefficients. Since $\mathrm{Tate}(q^M)$
is defined over $\Q((q))$, the point $\sigma(x)$ is
$(\mathrm{Tate}(q^M),\sigma\circ\iota\circ\sigma^{-1})$, and
\[
 (\sigma\circ\iota\circ\sigma^{-1})(z,\zeta_M^j)=zq^{t_\sigma^{-1}j}
 =\iota\bigl(z,\zeta_M^{t_\sigma^{-1}j}\bigr)
\]
for $z\in\mu_{2^h}$. Thus $\sigma(x)=\diam{t_\sigma^{-1}}x$. Since $\Xmu$ is
separated, the extensions of $\sigma(x)$ and $\diam{t_\sigma^{-1}}x$ to
$K[[q]]$-points coincide. Reducing at $q=0$, we get
$\sigma(\kappa)=\diam{t_\sigma^{-1}}\kappa$.

Now, as $\psi_\mu$ is defined over $\Q$,
\[
 \sigma(P)=\psi_\mu(\diam{t_\sigma^{-1}}\kappa)=P+\tau_M(t_\sigma^{-1})
 =P+\tau_M(t_\sigma),
\]
where the last equality holds because $\tau_M$ takes values in $E[2]$. In
particular $\sigma(P)=P$ if $\sigma$ fixes $K$, so $P\in E(K)$. Finally,
$2P=\varphi_0(\varpi(\kappa))$ is the image under $\pi_0$ of the class of a
difference of two cusps of $X_0(N)$, which is torsion by the theorem of
Manin and Drinfeld \cite{Manin1972,Drinfeld1973}.\footnote{In the present
case the argument is short. Let $p$ be a prime with $p\equiv1\pmod N$. The
Hecke correspondence $T_p$ on $X_0(N)$ sends the class of $s\in\Hup^*$ to the
sum of the classes of $(s+j)/p$, $0\leq j<p$, and $ps$. Since
$p\equiv1\pmod N$, for $s\in\mathbb P^1(\Q)$ one checks directly that these
$p+1$ points are $\Gamma_0(N)$-equivalent to $s$, so $T_p$ multiplies every
divisor supported on the cusps by $p+1$. On the other hand
$\pi_0\circ T_p=a_p\pi_0$, and $a_p=p+1-\#\widetilde E(\F_p)$ by the
Eichler--Shimura relation, where $\widetilde E$ is the reduction of $E$
modulo $p$. Applying $\pi_0$ to $T_p([\varpi(\kappa)]-[\infty])$ gives
$\#\widetilde E(\F_p)\cdot2P=0$, and $\#\widetilde E(\F_p)\geq1$.} Thus $P$
is a torsion point. This completes the proof of the proposition.
\end{proof}

\subsection{\texorpdfstring{Torsion points over unramified extensions of $\Q_2$}{Torsion points over unramified extensions of Q2}}\label{subsec:app-torsion}

In this subsection, we prove that $E$ has no point of order $4$ over
unramified extensions of $\Q_2$ (Proposition~\ref{prop:app-torsion}). Roughly
speaking, we write $E$ over $\Q_2$ in a normal form adapted to its points of
order $2$ (see \eqref{eq:app-normal} below). In most cases, an elementary
criterion (Lemma~\ref{lem:app-halving}) shows directly that no point of order
$2$ is divisible by $2$ over an unramified extension. In the
remaining case, we apply Theorem~\ref{thm:lifting} to a curve $2$-isogenous to
$E$, and the resulting rational points on that curve restrict the normal form
of $E$.

For an isogeny $\lambda:A\to A'$
of elliptic curves over $\Q$, define $m_\lambda$ by
$\lambda^*\omega_{A'}=m_\lambda\omega_A$, where $\omega_A$ and $\omega_{A'}$
are N\'eron differentials. By the N\'eron mapping property $m_\lambda\in\Z$,
and $m_\lambda m_{\lambda^\vee}=\deg\lambda$ since
$\lambda^\vee\lambda=[\deg\lambda]$. Over $\Q_2$ we use the same notation
with differentials of minimal Weierstrass equations; then $m_\lambda\in\Z_2$
is well defined up to a unit. In particular, if $\deg\lambda=2$, then
$v_2(m_\lambda)\in\{0,1\}$, and for $\lambda$ over $\Q$ we have
$m_\lambda=\pm2$ if and only if $v_2(m_\lambda)=1$. The following lemma is the
only place where Theorem~\ref{thm:lifting} enters this subsection.

\begin{lemma}\label{lem:app-isogeny}
Let $\lambda:E\to A$ be an isogeny of degree $2$ over $\Q$ with
$m_\lambda=\pm2$. Then every point $R\in A(\Qbar)$ with
$2R\in\lambda(E[2])$ is rational. In particular, $A[2]\subseteq A(\Q)$ and
$\lambda(E[2])\subseteq2A(\Q)$.
\end{lemma}

\begin{proof}
Replacing $\lambda$ by $-\lambda$, we may assume $m_\lambda=2$. Let
$\nu=\lambda^\vee:A\to E$. Then $m_\nu=1$, that is, $\nu^*\omega_E=\omega_A$.
Since $[2]\circ\psi_\mu=\varphi_0\circ\varpi$ and
$\varphi_0^*\omega_E=2\,\omega_f$, we have $\psi_\mu^*\omega_E=\omega_f$. By
Proposition~\ref{prop:integrality}\ref{int:etale}, with $V=\Xmu$, and
Theorem~\ref{thm:lifting}\ref{lift:rational}, with $\Gamma=\Gpm$, there is a
morphism $\eta:\Xmu\to A$ over $\Q$ with $\nu\circ\eta=\psi_\mu$.

Let $R\in A(\Qbar)$ with $2R\in\lambda(E[2])=\ker\nu$. Then $2\nu(R)=0$, so
$\nu(R)=\tau_\delta$ for some $\delta$. Note that the cusp $\diam\delta\infty$
is rational and $\eta$ is defined over $\Q$, so
$\eta(\diam\delta\infty)\in A(\Q)$; moreover
$\nu(\eta(\diam\delta\infty))=\psi_\mu(\diam\delta\infty)=\tau_\delta$. Thus
$R-\eta(\diam\delta\infty)$ lies in $\ker\nu=\lambda(E[2])\subseteq A(\Q)$,
and $R$ is rational. This completes the proof of the lemma.
\end{proof}

\begin{lemma}\label{lem:app-halving}
Let $F$ be a field of characteristic $0$, and let
$B:y^2=(x-e_0)(x-e_1)(x-e_2)$ with distinct $e_i\in F$. If
$(e_i,0)\in2B(F)$, then $e_i-e_j$ is a square in $F$ for both $j\neq i$.
\end{lemma}

\begin{proof}
Let $R\in B(F)$ with $2R=(e_i,0)$. Then $R\notin B[2]$, and the tangent line
at $R$ meets $B$ again at $-2R=(e_i,0)$, so it is $y=m(x-e_i)$ with
$m\in F^\times$. Writing $\{i,j,l\}=\{0,1,2\}$, we get
\[
 (x-e_i)(x-e_j)(x-e_l)-m^2(x-e_i)^2=(x-e_i)\bigl(x-x(R)\bigr)^2 .
\]
Cancelling $x-e_i$ and putting $x=e_j$ gives
$m^2(e_i-e_j)=(e_j-x(R))^2$, and similarly for $l$.
\end{proof}

Now we work over $\Q_2$. Let $B$ be an elliptic curve over $\Q_2$ with
$B[2]\subseteq B(\Q_2)$. Since $2$ is invertible in $\Q_2$, completing the
square in a Weierstrass equation of $B$ gives an equation $y^2=f(x)$ with
$f\in\Q_2[x]$ monic of degree $3$ \cite[Section~III.1]{Silverman2009}, and
$f$ has no repeated roots because $B$ is nonsingular. As $-(x,y)=(x,-y)$ on
this equation, the points of order $2$ of $B$ are the points $(e,0)$ with
$f(e)=0$. Hence the assumption $B[2]\subseteq B(\Q_2)$ means that $f$ splits
over $\Q_2$, and we can write $B$ as $y^2=(x-e_0)(x-e_1)(x-e_2)$ with distinct
$e_i\in\Q_2$. Put $w=\min_{i<j}v_2(e_i-e_j)$. Since the residue field of
$\Q_2$ is $\F_2$, the three differences $e_i-e_j$ do not all have valuation
$w$, and by the ultrametric inequality exactly one of them has valuation
greater than $w$. We number the roots so that this difference is $e_2-e_0$,
and put
\[
 a=2^{-w}(e_1-e_0)\in\Z_2^\times,\qquad b=2^{-w}(e_2-e_0),\qquad k=v_2(b)\geq1.
\]
Write $w=2s+r$ with $r\in\{0,1\}$ and $d=2^r$. The change of variables
$x-e_0=2^{2s}X$, $y=2^{3s}Y$ gives the equation
\begin{equation}\label{eq:app-normal}
 Y^2=X(X-da)(X-db),
\end{equation}
which we call the \emph{normal form} of $B$. We say that $B$ has \emph{type}
$(r,k)$, and call $T_B=(da,0)$ its \emph{isolated} point of order $2$.

In the remaining case, we need the multiplier of the isogeny with kernel
generated by the isolated point. It is computed from the normal form, which is
minimal in this case.

\begin{lemma}\label{lem:app-minimal}
Let $B$ have additive reduction and type $(0,k)$. Then
\eqref{eq:app-normal} is a minimal equation. If moreover
$a\equiv3\pmod4$, then $k\leq3$.
\end{lemma}

\begin{proof}
For $r=0$ the equation \eqref{eq:app-normal} has $c_4=16(a^2-ab+b^2)$
(see \cite[p.~42]{Silverman2009} for this notation), of valuation $4$. Since
\eqref{eq:app-normal} is integral, a minimal equation is obtained from it by a
change of variables with scaling factor $u$, $v_2(u)\geq0$, and has
$v_2(c_4)=4-4v_2(u)\geq0$ \cite[Table~3.1]{Silverman2009}. If $v_2(u)=1$, then
$c_4$ is a unit, so $B$ has good or multiplicative reduction
\cite[Prop.~VII.5.1]{Silverman2009}, contrary to the assumption.

If $a\equiv3\pmod4$ and $k\geq4$, the change of variables $X=4x_1$,
$Y=8y_1+4x_1$, with scaling factor $2$, transforms \eqref{eq:app-normal}
into
\[
 y_1^2+x_1y_1=x_1^3-\frac{a+b+1}{4}\,x_1^2+\frac{ab}{16}\,x_1,
\]
which is integral because $4\mid a+1$ and $16\mid b$. This contradicts the
minimality of \eqref{eq:app-normal}.
\end{proof}

\begin{lemma}\label{lem:app-isolated}
Let $B$ have additive reduction and type $(0,k)$ with $k\geq2$, and let
$\lambda:B\to B'$ be the $2$-isogeny with kernel generated by $T_B$. Then
$v_2(m_\lambda)=1$, and $B'$ has the integral equation
\begin{equation}\label{eq:app-target}
 Y'^2=X'\Bigl(X'^2-\bigl(a-\tfrac b2\bigr)X'+\tfrac{b^2}{16}\Bigr),
\end{equation}
in which $\lambda(B[2])=\{0,(0,0)\}$.
\end{lemma}

\begin{proof}
Put $x=X-a$ and $y=Y$. Then \eqref{eq:app-normal} becomes
$y^2=x\bigl(x^2+(2a-b)x+a(a-b)\bigr)$. The $2$-isogeny with kernel
$\langle(0,0)\rangle$ is
$(x,y)\mapsto(x',y')=\bigl(y^2/x^2,\,y(1-a(a-b)/x^2)\bigr)$, with target
\[
 y'^2=x'\bigl(x'^2-2(2a-b)x'+b^2\bigr)
\]
\cite[Example~III.4.5]{Silverman2009}. Since
$\dd x'=(1-a(a-b)/x^2)\dd x$, it pulls back $\dd x'/2y'$ to $\dd x/2y$. The two other points of order $2$ of
$B$ map to $(0,0)$. The substitution $x'=4X'$, $y'=8Y'$ gives
\eqref{eq:app-target}, which is integral because $k\geq2$. Since
$\dd X'/2Y'=2\,\dd x'/2y'$, we get $\lambda^*(\dd X'/2Y')=2\,\dd x/2y$. By
Lemma~\ref{lem:app-minimal}, $\dd x/2y=\dd X/2Y$ is a minimal differential of
$B$, and the minimal differential of $B'$ is a $\Z_2$-multiple of
$\dd X'/2Y'$ since \eqref{eq:app-target} is integral. Hence
$v_2(m_\lambda)\geq1$, and $v_2(m_\lambda)=1$ because
$v_2(m_\lambda)+v_2(m_{\lambda^\vee})=1$.
\end{proof}

Now we are ready to prove the main result of this subsection.

\begin{proposition}\label{prop:app-torsion}
Let $L$ be a finite unramified extension of $\Q_2$. Then
$E(L)[2^\infty]=E[2]$.
\end{proposition}

\begin{proof}
Since $E[2]\subseteq E(\Q)\subseteq E(L)$, and a point of order $2^n$ with
$n\geq2$ has a multiple of order $4$, it suffices to show that no point of
order $2$ lies in $2E(L)$. As $E[2]\subseteq E(\Q)$, the curve $E$ has a
normal form \eqref{eq:app-normal} of some type $(r,k)$ over $\Q_2$, with
roots $0$, $da$ and $db$. By Lemma~\ref{lem:app-halving}, it suffices to
find, for each root $e$, another root $e'$ such that $e-e'$ is not a square
in $L$. Since $L/\Q_2$ is unramified, this holds if $e-e'$ has odd valuation
or is congruent to $3$ modulo $4$. We find such roots in two cases.

\emph{Case 1: $r=1$, or $r=0$ and $k=1$.} If $r=1$, then $2a$ and $2a-2b$
have valuation $1$, and every root differs from another root by one of them. If $r=0$ and $k=1$, the roots $0$ and $b$ differ by $b$, of valuation $1$, and
since $b\equiv2\pmod4$, one of the differences $a-0$ and $a-b$ is congruent
to $3$ modulo $4$.

\emph{Case 2: $r=0$ and $k\geq2$.} This is the only case where the
assumption $c=2$ is used. We show that $k=3$ and $a\equiv3\pmod4$. Then the
roots $0$ and $b$ differ by $b$, which has valuation $3$, and the root $a$
differs from $0$ by $a\equiv3\pmod4$, as required. 

The isolated point $T_E$
is rational since $E[2]\subseteq E(\Q)$. Let
$\lambda:E\to A$ be the isogeny over $\Q$ with kernel $\langle T_E\rangle$.
By Lemma~\ref{lem:app-isolated}, $m_\lambda=\pm2$, and $A$ has the equation
\eqref{eq:app-target} over $\Q_2$, in which $\lambda(E[2])=\{0,(0,0)\}$. By
Lemma~\ref{lem:app-isogeny}, $A[2]\subseteq A(\Q)$ and $(0,0)\in2A(\Q)$.
Thus the quadratic factor in \eqref{eq:app-target}, which is monic with
coefficients in $\Z_2$, has roots $\alpha,\beta\in\Z_2$, and
Lemma~\ref{lem:app-halving} for the root $0$ of \eqref{eq:app-target} gives
$-\alpha=\rho^2$ and $-\beta=\theta^2$ with $\rho,\theta\in\Z_2$. Since
$\alpha\beta=b^2/16$, we may choose the sign of $\theta$ so that
$b=4\rho\theta$, and then $\alpha+\beta=a-b/2$ gives
\[
 a=-(\rho-\theta)^2,\qquad b=4\rho\theta .
\]
As $a$ is a unit, exactly one of $\rho,\theta$ is even. Hence
$a\equiv-1\pmod8$ and $k=2+v_2(\rho\theta)\geq3$, so $k=3$ by
Lemma~\ref{lem:app-minimal}.
\end{proof}

\begin{remark}\label{rem:app-not-local}
Proposition~\ref{prop:app-torsion} is not a purely local statement. For
example, the curve $y^2=x(x-1)(x-4)$ has additive reduction at $2$ and
rational $2$-torsion, but its point $(1,0)$ is twice the point
$(1+\sqrt{-3},\,3-\sqrt{-3})$, which is defined over $\Q(\sqrt{-3})$, and
$\Q(\sqrt{-3})$ is unramified at $2$. Its normal form is the equation itself,
of type $(0,2)$ with $a=1$, so it falls under Case~2 of the proof above, which
is excluded only through the assumption $c=2$.
\end{remark}

\subsection{\texorpdfstring{Proof that $c\neq2$}{Proof that c is not 2}}\label{subsec:app-proof}

In this subsection, we combine Propositions~\ref{prop:app-cusp}
and~\ref{prop:app-torsion} to prove the main result of this appendix.

\begin{theorem}\label{thm:app}
$c\neq2$.
\end{theorem}

\begin{proof}
We first show that $\tau_M\neq0$. Since $E[2]$ has exponent $2$, the
homomorphism $\tau$ kills all squares. An odd integer congruent to $1$
modulo $8$ is a square in $\Z_2$, hence a square modulo $2^h$. If $\tau_M=0$, then $\tau$
factors through $(\Z/N\Z)^\times\to(\Z/2^h\Z)^\times$, and, as $\tau(-1)=0$,
through $(\Z/2^h\Z)^\times/\langle-1,(\Z/2^h\Z)^{\times2}\rangle$. This group is
trivial for $h=2$ and a quotient of $(\Z/8\Z)^\times/\{\pm1\}$ for $h\geq3$,
so it has order at most $2$. This contradicts the surjectivity of $\tau$ onto
$E[2]$, which has order $4$. Thus $\tau_M\neq0$.

Now let $P$ and $t_\sigma$ be as in Proposition~\ref{prop:app-cusp}. Since
$\tau_M\neq0$ and $\sigma\mapsto t_\sigma$ maps $\operatorname{Gal}(\Qbar/\Q)$
onto $(\Z/M\Z)^\times$, we can choose $\sigma$ with $\tau_M(t_\sigma)\neq0$.
Let $n$ be an odd integer such that $nP$ has $2$-power order. As $\tau_M$
takes values in $E[2]$, we have
\[
 \sigma(nP)-nP=n\,\tau_M(t_\sigma)=\tau_M(t_\sigma)\neq0 .
\]
Thus $nP$ is not rational. In particular $nP\notin E[2]$, since
$E[2]\subseteq E(\Q)$. On the other hand, $nP\in E(K)$ and $K$ is unramified
at $2$, so $nP\in E(L)[2^\infty]$ for a finite unramified extension $L$ of
$\Q_2$. This contradicts Proposition~\ref{prop:app-torsion}. Hence the
assumption $c=2$ is false, and $c\neq2$.
\end{proof}
\endgroup

% !TEX root = ../paper.tex

\end{document}